\documentclass[
final
]{dmtcs-episciences}

\usepackage[utf8]{inputenc}
\usepackage{subfigure}

\usepackage{amsmath}
\usepackage{amsthm,thm-restate}
\usepackage{amssymb}
\usepackage{multirow,mathtools,booktabs}
\usepackage{xspace} 
\usepackage{verbatim}
\usepackage{hyperref}

\usepackage{booktabs}
\usepackage{tikz}
\usetikzlibrary{positioning}
\usetikzlibrary{shapes}
\usetikzlibrary{fit}
\tikzstyle{vert} = [circle,fill=black, minimum size=2mm, inner sep=0pt]
\tikzstyle{edge} = [thick]

\usepackage{tkz-euclide}
\usepackage{xstring}

\newenvironment{poc}{\begin{proof}[of Claim]}{\end{proof}}

\newtheorem{theorem}{Theorem}[section]
\newtheorem{proposition}[theorem]{Proposition}
\newtheorem{lemma}[theorem]{Lemma}

\newtheorem{corollary}[theorem]{Corollary}

\newtheorem{remark}[theorem]{Remark} 
\newtheorem*{clm}{Claim}

\newcommand{\tww}{\mathsf{tww}}

\newcommand{\cart}[2]{#1\,\square\, #2}
\newcommand{\lex}[2]{#1\,\circ \, #2}
\newcommand{\gmod}[2]{#1\,\diamond \, #2}
\newcommand{\rooted}[2]{#1\,\sharp \, #2}
 
\newcommand{\sub}{\operatorname{Sub}}

\newcommand{\cor}[2]{#1\,\bigcirc \, #2}

\newcommand*\circled[1]{\tikz[baseline=(char.base)]{
		\node[shape=circle,draw,inner sep=1pt] (char) {#1};}}
\DeclareMathOperator{\zigzag}{\circled{{\rm z}}}

\DeclareMathOperator{\replace}{\circled{{\rm r}}}

\usepackage[round]{natbib}

\author[William Pettersson and John Sylvester]{William Pettersson\affiliationmark{1}\thanks{Supported by Engineering and Physical Sciences Research Council (ESPRC) grant number EP/T004878/1.}
  \and John Sylvester\affiliationmark{1,2}\thanks{Supported by Engineering and Physical Sciences Research Council (ESPRC) grant number EP/T004878/1.}}
\title[Bounds on the Twin-Width of Product Graphs]{Bounds on the Twin-Width of Product Graphs}
\affiliation{
  School of Computing Science, University of Glasgow, UK\\
  Department of Computer Science, University of Liverpool, UK}
\begin{document}

\publicationdata{vol. 25:1}{2023}{18}{10.46298/dmtcs.10091}{2022-09-28; 2022-09-28; 2023-03-10}{2023-05-17}

\maketitle
\begin{abstract}
	
Twin-width is a graph width parameter recently introduced by Bonnet, Kim, Thomass\'{e} \& Watrigant. Given two graphs $G$ and $H$ and a graph product $\star$, we address the question: is the twin-width of $G\star H$ bounded by a function of the twin-widths of $G$ and $H$ and their maximum degrees? It is known that a bound of this type holds for strong products (Bonnet, Geniet, Kim, Thomass\'{e} \& Watrigant; SODA 2021).

We show that bounds of the same form hold for Cartesian, tensor/direct, corona, rooted, replacement, and zig-zag products. For the lexicographical product it is known that the twin-width of the product of two graphs is exactly the maximum of the twin-widths of the individual graphs (Bonnet, Kim,  Reinald, Thomass\'{e} \& Watrigant; IPEC 2021). In contrast, for the modular product we show that no bound can hold. In addition, we provide examples showing many of our bounds are tight, and give improved bounds for certain classes of graphs.
\end{abstract}

\section{Introduction}Twin-width is a recently introduced graph parameter which, roughly speaking, measures how much the neighbourhoods differ as pairs of (not necessarily adjacent) vertices are iteratively contracted until a single vertex remains~\cite{TWI}. Twin-width has already proven very useful in parameterised complexity. In particular, given a contraction sequence as input, first-order model checking is $\mathsf{FPT}$ in the formula length on classes of bounded twin-width~\cite{TWI}. Additionally, for graphs of bounded twin-width with the contraction sequence as input, there is a linear time algorithm for triangle counting \cite{kratsch2022triangle}, and an $\mathsf{FPT}$-algorithm for maximum independent set \cite{TWIII}.  Polynomial kernels for several problems \cite{Bonnet0RTW21} are also known which do not require the contraction sequence to be given as input. However, deciding if the twin-width of a graph is at most four is $\mathsf{NP}$-complete \cite{Hardness}. Bounds on twin-width are known for many graph classes~\cite{BalabanH21,TWII,TW:IV,perms,g2021stable,jacob2022bounding,simon2021ordered}.
In particular, twin-width is bounded on some classes of dense graphs (e.g.,~the class of complete graphs has twin-width zero) and on some classes of sparse graphs (e.g.,~the class of grid graphs has bounded twin-width), and indeed both graph classes of bounded tree-width and graph classes of bounded rank-width have bounded twin-width \cite{TWI,jacob2022bounding}. Several other aspects of twin-width, such as $\chi$-boundedness, colouring numbers and VC-density, have recently received attention  \cite{DreierGJMR22,TWVI,pilipczuk2022graphs,pilipczuk2021compact,przybyszewski}.

We study the behaviour of the twin-width under graph products, which are well-known tools for constructing well-structured graphs. Understanding how twin-width behaves under graph products can give useful insights into the construction of classes of graphs with either bounded or unbounded twin-width, depending on the application. For instance, it has been shown that the twin-width of the strong product of two graphs is bounded by a function of their respective twin-widths and maximum degrees~\cite{TWII}. The twin-width of wreath products of groups has also been studied \cite{Groups}. 
For graph products where the twin-width is well-behaved, we will see is most but not all of them, our results are useful for creating examples of new classes of bounded twin-width, for example see Theorem \ref{thm:bonext}. 

We now briefly introduce the graph products studied in this paper, formal definitions appear at the start of the section named after each product. We begin with the Cartesian $\square$, tensor $\times$, strong $\boxtimes$, and lexicographic $\circ$ products which are the `\emph{four standard}' graph products~\cite{HahnTardif,Handbook}. These products are all widely studied and we provide some references which are related to our results in the next paragraph. We also consider the modular product~$\diamond$ which was first described in \cite{ImrichAssociative}, then rediscovered several times \cite{BarrowMod,LeviMod,VizingMod} in connection to the subgraph isomorphism problem, see also \cite{PikeMod,RaymondMod,ShaoMod}. In addition, we study the corona product  $\bigcirc $ which was originally introduced in~\cite{Frucht1970-tk} and has since been studied in several contexts including spectral properties \cite{BarikPS07}, $k$-domination \cite{ChellaliFHV12} and total-colouring \cite{MohanGS17}.  Similarly the rooted product $\sharp $, was introduced in~\cite{GodsilMcKay} for its spectral properties. This product can be seen as one layer of a Cartesian product and has been studied in several contexts \cite{KohRooted,JakovacRooted,RosenfeldRooted}. The above products can be taken between any two graphs $G$ and $H$, we also study the zig-zag $\zigzag$ and replacement $\replace$ products which are only defined for regular graphs. The zig-zag product was introduced in~\cite{Reingold} for the efficient construction of bounded degree expanders, the replacement product has been around longer but has also been used for this purpose \cite{Alon}.

The question of whether, given a graph parameter $w$ and a graph product $\star$, $w(G\star H)$ is bounded by a function of $w(G)$ and $w(H)$ is natural and has been studied many times. One example is Hedetniemi's conjecture \cite{Hedetniemi} which states that for the chromatic number $\chi$, the tensor product $\times$, and any graphs $G,H$, we have  $\chi(G\times H)= \min\{\chi(G),\chi(H) \}$. This conjecture was recently proved false~\cite{Shitov}, by the description of a pair of graphs $G$ and $H$ satisfying $\chi(G\times H)< \min\{\chi(G),\chi(H) \}$. Interestingly if one replaces the tensor product $\times$ with the Cartesian product $\square$ then Hedetniemi's conjecture holds \cite{Sabidussi}, that is $\chi(\cart{G}{H})= \min\{\chi(G),\chi(H) \}$. Vizing's conjecture \cite{Vizing} on the other hand, is of a similar vintage (posed in 1968) and is still open. The conjecture states that if $\gamma(G)$ is the size of the smallest dominating set in $G$, then for any graphs $G$ and $H$ we have $ \gamma(G\square H ) \geq \gamma(G)\cdot \gamma(H)$. This conjecture is known to be true up to a multiplicative factor of a half \cite{ClarkS00,SuenT12}, and known to be false for tensor products \cite{KlavzarVizConj}. Additionally, if $\alpha(G)$ is the size of the largest independent set then for the lexicographic product then $\alpha(\lex{G}{H})= \alpha(G)\cdot \alpha(H)$ holds for any graphs $G$ and $H$ \cite{GELLER197587}. It is known \cite[Lemma 9]{Bonnet0RTW21} (also see Theorem \ref{thm:lex} for an alternative proof) that the lexicographic product is similarly well behaved with respect to twin-width. The clique-width of various graph products has also been studied \cite{Gurski17}.

\subsection{Our Results}

As mentioned above the main question we address is, given a graph product $\star$, what can be said of $\tww(G \star H)$ for any graphs $G$ and $H$? Our main finding is that for all graphs products studied (with the exception of the modular product) the twin-width of a product of two graphs is functionally bounded by the twin-width and maximum degrees of the individual graphs. These upper bounds for the twin-width of graph products are summarised in Table \ref{tbl:results}. Note that twin-width is not monotone with respect to non-induced subgraphs, thus none of these bounds are necessarily implied by any of the others.

\begin{table}[ht] \begin{tabular*}{\textwidth}{l   c  l}\toprule
		Product Name $\star$ & Upper Bound on $\tww(G \star H)$ &  Reference  \\ \midrule
		\addlinespace[1\defaultaddspace]
		
		Cartesian  $\square$& $\max\{\tww(G) +\Delta(H), \;\tww(H)  \} + \Delta(H)$ & Theorem \ref{thm:carti}  \\ \addlinespace[1\defaultaddspace]
		
		Tensor  $\times$ & $ \max\big\{\tww(G)\cdot \Delta(H) + \Delta(H), \;\tww(H)\big\}  + \Delta(H)$ & Theorem \ref{thm:tensor}  \\ \addlinespace[1\defaultaddspace]
		
		Strong  $\boxtimes$ & $ \max\left\{\tww(G)(\Delta(H)+1)+\Delta(H),\; \tww(H)\right\}+\Delta(H)$ &  \cite{TWII}  \\ \addlinespace[1\defaultaddspace]

		Lexicographic  $\circ$ & $\max\{\tww(G), \tww(H)\}$ \quad (\textbf{equality})  & \cite{Bonnet0RTW21}   \\ \addlinespace[1\defaultaddspace]

		Modular  $\diamond$ & \textbf{Not bounded by any} $f(\tww(G), \tww(H),\Delta(G),\Delta(H))$ & Theorem \ref{thm:mod-lb}    \\ \addlinespace[1\defaultaddspace]
		Corona  $\bigcirc $ &  $ \max\{\tww(G)+1, \tww(H), 2\}$ & Theorem \ref{thm:corona}   \\ \addlinespace[1\defaultaddspace]
		Rooted  $\sharp$ & $ \max \{ \tww(H) + 1, \Delta(H), \tww(G) + 1, 2\}$ & Theorem \ref{thm:rooted}    \\ \addlinespace[1\defaultaddspace]
		
		Replacement  $\replace$ & $    \tww(G) + \Delta(G)$ & Theorem \ref{replace}   \\ \addlinespace[1\defaultaddspace]
		
		Zig-Zag  $\zigzag$ & $\max\left\{ \Delta(H)^2\cdot (\Delta(G) -\Delta(H) +1) ,\; \tww(G)+ \Delta(G) \right\}$ & Theorem \ref{zigzag}   \\ \addlinespace[1\defaultaddspace]

		\bottomrule
	\end{tabular*}\caption{Bounds on the twin-width of product graphs, see relevant sections for definitions of the graph products. The lexicographic bound is an equality, i.e.\ $\tww(\lex{G}{H})= \max\{\tww(G),\tww(H)\} $.}
	\label{tbl:results}
\end{table} 
Using our upper bounds we can determine the twin-width of several families of graphs, for example rook graphs, Hamming graphs, and weak powers of cliques. We show tightness of our bounds on the twin-width of Cartesian, tensor and corona products. We also show tightness of the bound for strong and lexicographical products proved by \cite{TWII} and \cite{Bonnet0RTW21} respectively. We also show that our bound on rooted products is almost tight. We prove that twin-width of a modular product of two graphs cannot (in general) be functionally bounded by twin-width and maximum degrees of the individual graphs by showing that the twin-width of the modular product of two paths diverges. 

For the replacement product we prove tightness up-to a constant factor, however we leave open whether our bound for the zig-zag product is tight - we strongly suspect this can be improved. Additionally, we prove a lower bound for replacement product graphs which is a constant times the square root of their degree.  

It is known \cite[Theorem 2.9]{TWII} that if $\mathcal{C}$ is a hereditary class of graphs that has  bounded twin-width, then it is $K_{t,t}$ free if and only if the subgraph closure of $\mathcal{C}$, denoted by $\sub(\mathcal{C})$, has bounded twin-width. This fact can then be used in combination with a bound on the twin-width of strong products to establish a further bound on the twin-width of the subgraph closure of a strong products of two classes \cite[Theorem 2.8]{TWII}. Our bounds can be applied in the exact same way to extend this result, as follows. 

\begin{theorem}[extension of {\cite[Theorem 2.8]{TWII}}] \label{thm:bonext}
	Let $\star$ be one of the Cartesian $\square$, tensor $\times$, strong $\boxtimes$, lexicographic $\circ$, corona  $\bigcirc $, rooted $\sharp$, replacement $\replace$, or zig-zag $\zigzag$ products and let $\mathcal G$ and $\mathcal H$ two classes such that $\mathcal G \star \mathcal H$ is $K_{t,t}$-free. Then there exists some function $f$ such that  \begin{equation*}\tww(\sub(\mathcal G \star \mathcal H)) \leqslant f(\tww(\mathcal G),\;\tww(\mathcal H),\;\Delta(\mathcal G),\;\Delta(\mathcal H),\;t).\end{equation*}
\end{theorem}

\subsection{Organisation of the Paper}
We begin in Section \ref{sec:notation} with some notation and definitions. Sections \ref{sec:standrd} \& \ref{sec:nonstandrd} contain our findings for the twin-width of each graph product. Each subsection of Sections \ref{sec:standrd} \& \ref{sec:nonstandrd} begins with the definition of a graph product $\star$ before stating our results for it. We finish with some conclusions in Section \ref{sec:conclude}. 

\section{Notation and Definitions}\label{sec:notation}
\paragraph{Elementary Definitions.} Throughout the paper we assume all graphs are simple, that is they have no loops or multiple edges. Given a graph $G$, the neighbourhood of a vertex $v\in V(G)$, written $N_G(v)$, is the set $\{u : uv \in E(G)\}$.
The degree of a vertex $v$ in $V(G)$, written $d_G(v)$, is the size of $N_G(v)$ (i.e.,~$d_G(v) = |N_G(v)|$). We drop subscripts when the graph is clear from the context. We let $\Delta(G)=\max_{v\in V(G)}d_G(v)$ denote the maximum degree of a graph $G$ taken over all vertices in $V(G)$. The graph distance between two vertices $u,v\in V(G)$ is denoted by $d_G(u,v)$. We use $K_n$ to denote a clique on $n$ vertices, $K_{n,m}$ to denote a complete bipartite graph with bags of size $n$ and $m$, and $P_n$ to denote a path on $n$ vertices. For two graphs $G$ and $H$, we write $G \cong H$ to mean that $G$ is isomorphic to $H$, and for two sets $S,K$ we let $S\triangle K = (S \setminus K) \cup (K \setminus S)   $ denote their symmetric difference. 

\paragraph{Graph Products.} Let $G$ and $H$ be any two graphs, and let $V(G) \times V(H) = \{(u,i) : u\in V(G), i\in V(H) \}$ be the product of their vertex sets. Note that, to make proofs easier to follow, we will use letters from $\{u,v,w\}$ for vertices from $G$, and letters from $\{i,j,k\}$ for vertices from $H$.
We consider several graph products $G\star H$ in this paper, all will have the vertex set $V(G)\times V(H)$ but the edge sets will depend on the product $\star$ and the graphs $G$ and $H$. See the relevant sections for definitions of the different products.

\paragraph{Trigraphs, contraction sequences, and twin-width of a graph.} Following the notation introduced in \cite{TWI}, a \emph{trigraph $G$} has vertex set $V(G)$, a set $E(G)$ of black edges, and a set $R(G)$ of red edges (which are thought of as `error edges'), with $E(G)$ and $R(G)$ being disjoint.
In a trigraph $G$, the neighbourhood $N_G(v)$ of a vertex $v\in V(G)$ consists of all the vertices adjacent to $v$ by a black or red edge. A $d$-trigraph is a trigraph $G$ such that the \emph{red graph} $(V(G),R(G))$ has degree at most~$d$. In that instance, we also say that the trigraph has \emph{red degree} at most~$d$.

A \emph{contraction} in a trigraph~$G$ consists of merging/contracting two (non-necessarily adjacent) vertices $u$ and $v$ into a single vertex $w$, and updating the edges of $G$ as follows. Every vertex of the symmetric difference $N_G(u) \triangle N_G(v)$ is linked to $w$ by a red edge, and every vertex $x$ of the intersection $N_G(u) \cap N_G(v)$ is linked to $w$ by a black edge if both $ux \in E(G)$ and $vx \in E(G)$, and by a red edge otherwise. All other edges in the graph (those not incident to $u$ or $v$) remain unchanged. The vertices $u$ and $v$, together with any edges incident to these vertices, are removed from the trigraph. To make proofs easier to follow, we will sometimes use $v$ to refer to the vertex $w$ if $v$ has been removed as part of contracting $u$ and $v$ into $w$. That is, one can think of the new vertex $w$ as having both the label $v$ and the label $u$.

A \emph{$d$-sequence} is a sequence of $d$-trigraphs $G_n, G_{n-1}, \ldots, G_1$, where $G_n = G$, $G_1=K_1$ is the graph on a single vertex, and $G_{i-1}$ is obtained from $G_i$ by performing a single contraction of two (non-necessarily adjacent) vertices. Note that $G_i$ has precisely $i$ vertices, for every $i \in [n]$.
The twin-width of $G$, denoted by $\tww(G)$, is the minimum integer~$d$ such that $G$ admits a~$d$-sequence.

We say that $G$ is a \emph{trigraph over a graph $H$} if $(V(G),E(G) \cup R(G))$ is isomorphic to $H$. 

\begin{lemma}[{\cite[Lemma 7.1]{TWII}}]\label{lem:degreetww}
	Every trigraph over a graph $H$ has twin-width at most $\tww(H)+\Delta(H)$.
\end{lemma}

The following lower bound is used many times in this paper.
\begin{lemma}[{\cite[Lemma 3.1]{AhnBounds}}]\label{lowerbound} Any graph $G$ satisfies \begin{equation*}\tww(G) \geq \min\{|(N(v)\triangle N(u))\backslash\{u,v\}| : u,v\in V(G), \; u\neq v \}.\end{equation*}
\end{lemma}  
We also use the following useful fact. 
\begin{proposition}[{\cite[Section 4.1]{TWI}}]\label{prop:induced} The twin-width of an induced subgraph $H$ of a trigraph $G$ is at most the twin-width of $G$.
\end{proposition}

\section{Bounds on Twin-width of Standard Products}\label{sec:standrd}

\subsection{Cartesian Product}

The \emph{Cartesian product} $\cart{G}{H}$ has vertex set $V(G) \times V(H)$ and
\begin{equation*} 
	(u,i)(v,j) \in E( \cart{G}{H})\quad  \text{if and only if}\quad [u = v \text{ and }ij\in E(H)]\text{ or }[i = j\text{ and }uv\in E(G)].
\end{equation*} 

\begin{restatable}{theorem}{carti}\label{thm:carti} For any graphs $G$ and $H$ we have 
	\begin{equation*} \tww\left(\cart{G}{H}\right)\leq \max\{\tww(G) +\Delta(H), \;\tww(H)  \} + \Delta(H).\end{equation*} 
\end{restatable}
\begin{proof} Let $G=G_n, \dots, G_1=K_1$ be a $\tww(G)$-sequence for $G$, and label the $m$ vertices of $H$ by $[m]$ in any order. For a fixed $i \in [m]$, we call \emph{$i$-th copy of $G$}, the subgraph of $\cart{G}{H}$ induced by vertices of the form $(v,i)$ for every $v \in V(G)$. We continue to refer to all vertices $(v,i)$ as the $i$-th  copy of $G$ even if contractions have been made. 
	
	We contract $\cart{G}{H}$ to $\cart{K_1}{H}$ by following the series $G_n, \dots, G_1$ of contractions in each copy of $G$. That is, we reduce $G_n$ to $G_{n-1}$ in the first copy, then the second, etc. until we have reduced $G_n$ to $G_{n-1}$ in the $m$-th  copy, and then we return to the first copy to reduce $G_{n-1}$ to $G_{n-2}$ etc.\ We proceed like this until we are left with a trigraph over $\cart{K_1}{H}$, which is isomorphic to $H$. We then reduce this trigraph using a $\tww(H)$-sequence for $H$. 
	
	Observe that at any stage in the contraction sequence of $\cart{G}{H}$ described above there can be at most $\tww(G)$ red edges of the form $(v,i)(u,i)$, for some fixed $v\in V(G)$, in the current trigraph. This follows since we do not make any contractions between two vertices $(u,i)$ and $(v,j)$ where $j\neq i$ until there is only a single vertex left in each copy of $G$. Now observe that there are at most $2\Delta(H)$ edges of the form $(v,i)(u,j)$, where $j\neq i$, for any $v\in V(G)$ at any given time. To see this note that if $(u,i)$ and $(v,i)$ are being contracted to $(w,i)$ in the $i$-th  copy of $G$ but $(u,j)$ and $(v,j)$ have not yet been contracted in some other copies of $G$ where $j\neq i$, then the red edges $(w,i)(u,j)$ and $(w,i)(v,j)$ will be formed for each such $j$. There can be at most $2\Delta(H)$ such edges and so it follows that there is a $\tww(G)+2\Delta(H)$-sequence transforming $\cart{G}{H}$ into a trigraph over $\cart{K_1}{H}\cong H$. This trigraph admits a $\tww(H)+\Delta(H)$-sequence by Lemma \ref{lem:degreetww}. Thus $\tww(\cart{G}{H}) \leq \max\{\tww(G) + 2\Delta(H), \;\tww(H) + \Delta(H)  \} $ as claimed. \end{proof}

As $\cart{K_2}{K_2} \cong C_4$, which has twin-width $0$, Theorem \ref{thm:carti} does not always give a tight bound. However the following result shows that Theorem \ref{thm:carti} is tight for Cartesian products of two non-trivial cliques where at least one has more than two vertices. 

\begin{restatable}[Rook Graphs]{proposition}{rookgraphs}\label{prop:cartclique}For any $n,m\geq 1$ we have \begin{equation*}\tww(K_n\square K_m)=\begin{cases} 0 &\text{ if }n=m=2\text{ or } \min\{n,m\}=1\\
			2(\min\{n,m\}-1) & \text{ otherwise}
		\end{cases}.\end{equation*} 
\end{restatable}
\begin{proof}Observe that $K_2\square K_2\cong C_4$ and $K_1 \square K_n \cong K_n$ for any $n\geq 1$. Thus as we have $\tww(C_4)$ and $\tww(K_n)=0$, we establish the first case. We can now assume that $n,m\geq 2$ and $(n,m)\neq (2,2)$. 
	
	For the upper bound observe that as $\cart{K_n}{K_m} \cong \cart{K_m}{K_n}$ the bound \begin{equation*}\tww(\cart{K_n}{K_m})\leq \max\{0+ \Delta(K_m),0 \} + \Delta(K_m) = 2\cdot \Delta(K_m),  \end{equation*}and also the bound $\tww(\cart{K_m}{K_n})\leq 2\cdot \Delta(K_n) $ both hold by Theorem \ref{thm:carti}, as claimed.
	
	For a graph $G$ and $u,v \in V(G) $ let $b(u,v) =  |(N(v)\triangle N(u))\backslash\{u,v\}|.$ The lower bound follows from Lemma \ref{lowerbound} by considering $b(u,v)$ for all pairs of distinct vertices $(u,i),(v,i)\in V(K_n\square K_m)$. We have two cases:
	
	\noindent\textbf{Case 1 [$u=v$ or $i=j$]:} For any two vertices $u\neq v$ we have \begin{equation*}(N((u,i))\triangle N((v,i)))\setminus \{(u,i),(v,i)\} = \left\{(z,j): z\in \{u,v\}, j\in [m]\setminus \{i\} \right\}, \end{equation*} thus $b((u,i),(v,i))\geq 2(m-1)$. By symmetry we have $b((u,i),(u,j))\geq 2(n-1)$ for any $i\neq j$ and $u\in V(G)$.

	\noindent\textbf{Case 2 [$v\neq u$ and $i\neq j$]:}            For any $v\neq u$ and $i\neq j$ we have \begin{equation*}\begin{aligned}&(N((u,i))\triangle N((v,j)))\setminus \{(u,i),(v,j)\}\\ &\qquad= \left\{(z,k): z\notin \{u,v\}, \;k\in \{i,j\} \right\}\cup \left\{(z,k): z\in \{u,v\}, \; k\in [m] \setminus \{i,j\} \right\}, \end{aligned} \end{equation*}
	thus $b((u,i),(v,j)) \geq 2(n-2) + 2(m-2) \geq 2(\min\{m,n\} -1)$, as $\max\{m,n\}\geq 3$. 
\end{proof}

For integers $d,k\geq 1 $, let $ \mathbb{H}(d,k)$ be the \emph{Hamming} graph with $k^d$ vertices, defined inductively by \begin{equation*}  \mathbb{H}(d,k) = \cart{K_k}{  \mathbb{H}(d-1,k)}, \qquad \text{and}\qquad \mathbb{H}(1,k)=K_k.\end{equation*} Hamming graphs are a generalisation of Rook graphs and have applications in coding theory and distributed computing. Arguably the most well known Hamming graph is the \emph{hypercube} $ \mathbb{H}(d,2)$. Our next result determines the twin-width of all Hamming graphs and is proved by applying Theorem \ref{thm:carti} iteratively. Similarly to Rook graphs, the twin-width of Hamming graphs matches the bound in Theorem \ref{thm:carti}.    
\begin{restatable}[Hamming Graphs]{proposition}{hamming}\label{prop:hamming}For any $d,k\geq 1$ we have  \begin{equation*}\tww(  \mathbb{H}(d,k))= \begin{cases} 0 &\text{ if } d=1 \text{ or }k=1\\ 
			2(k-1)(d-2) & \text{ if } d\geq 2 \text{ and }k=2\\
			2(k-1)(d-1) & \text{ if } d\geq 2 \text{ and } k\geq 3
		\end{cases}.\end{equation*}
\end{restatable}	

\begin{proof} The case $d\leq 2$ is covered by Proposition \ref{prop:cartclique}. Observe that if $d,k\geq 2$ then we can express the function in the statement as $t(d,k)=2(k-1)(d-1 -\mathbf{1}(k=2))$.  
	
	Applying Theorem \ref{thm:carti} iteratively to $  \mathbb{H}(d,k)$, for $d\geq 3$, and using the fact that $\tww(K_k)=0$ gives
	\begin{equation*}	\begin{aligned} \tww(  \mathbb{H}(d,k))&= \tww( \cart{  \mathbb{H}(d-1,k)}{K_k})\\ &\leq  \max\{\tww(  \mathbb{H}(d-1,k)) +k-1, 0 \} + k-1 \\& = \tww(  \mathbb{H}(d-1,k)) + 2(k-1) \\
			&\leq \cdots \\
			&\leq \tww(  \mathbb{H}(2,k)) + (d-2)\cdot 2(k-1).  \end{aligned}
	\end{equation*}  By Proposition \ref{prop:cartclique} we have $\tww(  \mathbb{H}(2,2))=0$ and $\tww(  \mathbb{H}(2,k))=2(k-1)$ if $k\geq 3$, thus in either case $\tww(  \mathbb{H}(d,k))\leq t(d,k)$.

	We now consider the lower bound for $d,k\geq 2$. The lower bound will follow Lemma \ref{lowerbound} by considering $b(u,v) =  |(N(v)\triangle N(u))\backslash\{u,v\}|$ for all pairs of distinct vertices $u,v$ in $  \mathbb{H}(d,k)$. It will help to consider $  \mathbb{H}(d,k)$ as follows: we identify each vertex as a string in $\{0,\dots, k-1\}^d$ where two vertices are connected if and only if their strings differ in one exactly coordinate. Note that $  \mathbb{H}(d,k)$ is $d(k-1)$-regular. We have three cases: 
	
	\noindent\textbf{Case 1 [$uv\in E(  \mathbb{H}(d,k))$]:} In this case $u$ and $v$ differ at exactly one coordinate, say $i$. Suppose there exists some $w\in N(u)\cap N(v)$ where $w\notin \{ u,v\}$. It follows that $w$ has the same value as both $u$ and $v$ at all but the $i$-th  coordinate. Since the $i$-th  coordinate of $w$ cannot match that of $u$ or $v$ this leaves $k-2$ other options. Thus $|N(u)\cap N(v)|=k-2 $, and so $b(u,v)= 2(d(k-1) -(k-2)-1) = 2(d-1)(k-1)$. Observe that if $k=2$ then $b(u,v)= 2d-2 > 2(d-2) =t(d,2)$ and if $k\geq 3$ then $b(u,v)= 2(d-1)(k-1) =t(d,k)$.  
	
	\noindent\textbf{Case 2 [$d_{  \mathbb{H}(d,k)}(u, v)=2$]:} It follows that $u$ and $v$ differ at two coordinates $i$ and $j$ where $i\neq j$. Thus if $w\in V(  \mathbb{H}(d,k))$ is adjacent to both $u$ and $v$ it must be equal to $u$ at all but one coordinate and equal to $v$ at all but one coordinate, so $|N(u)\cap N(v)|=2$. Since $uv\notin E$ we have $b(u,v)= 2(d(k-1)-2)$. Observe that if $k=2$ then $b(u,v)= 2(d-2) =t(d,2)$ and if $k\geq 3$ then $b(u,v)=2d(k-1)-4 \geq 2(d-1)(k-1) =t(d,k)$.  
	
	\noindent\textbf{Case 3 [$d_{  \mathbb{H}(d,k)}(u, v)\geq 3$]:} Since no pair of vertices at distance $3$ or greater can have any shared neighbours, we have $b(u,v)= 2d(k-1)>t(d,k)$. 
	
	Thus by Lemma \ref{lowerbound} we have $\tww(  \mathbb{H}(d,k)) \geq t(d,k)$ as claimed.  \end{proof}

\subsection{Tensor Product}\label{sec:tensor}

The \emph{tensor product} $ {G} \times {H}$ (also called the direct
product, Kronecker product, weak product, or conjunction) has vertex set $V(G) \times V(H)$ and
\begin{equation*} 
	(u,i)(v,j) \in E( {G} \times {H})\quad  \text{if and only if}\quad [ij\in E(H)]\text{ and }[uv\in E(G)].
\end{equation*}

\begin{restatable}{theorem}{tensor}\label{thm:tensor} For any graphs $G$ and $H$ we have 
	\begin{equation*}\tww\left({G}\times {H}\right)\leq \max\big\{\tww(G)\cdot \Delta(H) + \Delta(H), \;\tww(H)\big\}  + \Delta(H) .\end{equation*} 
\end{restatable}
\begin{proof}Given a $\tww(G)$-sequence $G_n, \dots, G_1$ of contractions for $G$ the first stage is to apply one contraction at a time in each copy (as we did in the proof of Theorem \ref{thm:carti}), which leaves a trigraph over $H$. The second stage, as in Theorem \ref{thm:carti}, is to reduce this trigraph using a $\tww(H)$-sequence for $H$. 
	
	To begin observe that there are no red edges of the form $(v,i)(u,i)$ created in the first stage. This follows since there are no edges of the form $(u,i)(v,i)$ in $E({G}\times{H})$, we only perform contractions between two vertices $(u,i)$ and $(v,i)$, and edges of this form cannot be created from contractions. We claim that the largest red degree in the first stage is at most $(\tww(G)+2)\cdot \Delta(H)$. First consider the red degree of a vertex $(w,i)$ directly after a contraction of $(u,i)$ and $(v,i)$ into $(w,i)$. After the equivalent contraction in $G$ the vertex $w$ has at red degree at most $\tww(G)$, and all these edges are of the form $wz$ where $z\notin \{u,v\}$, thus these edges correspond to $\tww(G)\cdot\Delta(H)$ red edges of the form $(w,i)(z,j)$. However there may be up to $2\Delta(H)$ additional other red edges of the form $(w,i)(u,j)$ and $(w,i)(v,j)$ created by the contraction of $(u,i)$ and $(v,i)$. This occurs if there was an edge $uv\in E(G)$ and the pair $u,v$ we merged in $i$-th  copy before they were merged in some other copy $j$, where $ij\in E(H)$. Thus there is a $(\tww(G)+2)\cdot\Delta(H)$-sequence transforming  ${G}\times{H}$ into a trigraph over $H$. This trigraph admits a $\tww(H)+\Delta(H)$-sequence by Lemma \ref{lem:degreetww}, thus $\tww(\cart{G}{H}) \leq \max\{(\tww(G)+2)\cdot\Delta(H), \;\tww(H) + \Delta(H)  \} $ as claimed. \end{proof}

For a graph $G$ we let $\overline{G}$ denote its complement and note that for any graph $G$ we have $\tww(G)=\tww(\overline{G})$. For the special case of cliques, observe that $\overline{K_n\square K_m}=K_n\times K_m$. Thus we obtain the following corollary to Proposition \ref{prop:cartclique}.

\begin{corollary}[Rook Complement Graphs]\label{Cor:tensorClique}For any $n,m\geq 1$ we have \begin{equation*}\tww(K_n\times K_m)=\begin{cases} 0 &\text{ if }n=m=2\text{ or } \min\{n,m\}=1\\
			2(\min\{n,m\}-1) & \text{ otherwise}
		\end{cases}.\end{equation*} 
\end{corollary}

For integers $d,k\geq 1 $, let $ \mathbb{T}(d,k)$ be the \emph{weak power of a clique}, defined inductively by \begin{equation*} \mathbb{T}(d,k) = K_k \times  \mathbb{T}(d-1,k)\qquad \text{and} \qquad \mathbb{T}(1,k) = K_k .\end{equation*}
Weak powers of cliques have been studied in the contexts of colourings \& independent sets \cite{MR357199,MR2105948,MR2368031} and isoperimetric inequalities \cite{Brakensiek17}. Using Theorem \ref{thm:tensor} we can determine the twin-width for these graphs.

\begin{proposition}[Weak Powers of Cliques]\label{prop:tensorhamming}For any $d,k\geq 1$ we have  \begin{equation*}\tww(  \mathbb{T}(d,k))= \begin{cases} 0 &\text{ if } d=1 \text{ or }k\leq 2\\ 
			2(k-1)^{d-1} & \text{ if } d\geq 2 \text{ and } k\geq 3
		\end{cases}.\end{equation*}
\end{proposition}	

\begin{proof} The following observation holds by the definition of tensor product: $\mathbb{T}(d,k)$ is the graph on vertex set $[k]^d$ where two vertices $u, v \in [k]^d$ are connected if and only if $u_i \neq  v_i$ for all $i \in [k]$.
	
	To begin, if $ k=1$ or $d=1$ then the graph is either a set of isolated vertices or a single clique, respectively, so the result holds. For the case $k=2$ we claim that for any $d\geq 1$ the graph $\mathbb{T}(d,2)$ is a complete matching; thus $\tww(\mathbb{T}(d,2))=0$, in agreement with the statement. To see this note that, by the observation above, any vertex $x=(x_1, \dots x_d)\in V(\mathbb{T}(d,2))= \{0,1\}^d $ only has one neighbour given by  $\bar{x}=(1-x_1, \dots 1-x_d) $. 
	
	Thus from now on we can assume that $k\geq 3$ and $d\geq 2$.  
	
	We now prove the upper bound $ \tww(  \mathbb{T}(d,k))\leq 	2(k-1)^{d-1}$ by induction in $d$. Observe that for any $k\geq 3$ and $d=2$ we have $\tww(  \mathbb{T}(2,k)) =	2(k-1)$ by Corollary \ref{Cor:tensorClique}, which establishes the base case. For the inductive step observe that $\mathbb{T}(d,k) $ is a regular graph of degree $(k-1)^d$ since, for any fixed vertex $x$, each coordinate of of a neighbouring vertex $y\in [k]^d$ can take one of $k-1$ values which differ from the corresponding coordinate of $x$. Recall also that $\tww(K_k)=0$, thus by Theorem \ref{thm:tensor}, for any $k\geq 3$, we have 
	\begin{equation*}	\begin{aligned}\tww(  \mathbb{T}(d,k))&\leq  \max\{\tww(K_k)\cdot \Delta(\mathbb{T}(d-1,k)) +\Delta(\mathbb{T}(d-1,k)), \tww(  \mathbb{T}(d-1,k)) \} + \Delta(\mathbb{T}(d-1,k)) \\
			&\leq  \max\{0\cdot (k-1)^{d-1} +(k-1)^{d-1},\; 2(k-1)^{d-2} \} + (k-1)^{d-1}\\
			& = 2(k-1)^{d-1}.  \end{aligned}
	\end{equation*}

	For the lower bound, observe that any two vertices $u,v\in V(\mathbb{T}(d,k)))$, where  $u\neq v$, must differ at $i\in[d]$ coordinates. Thus for such a pair we have $N(u)\cap N(v) = (k-2)^i(k-1)^{d-i} $ as for a vertex $w$ to be adjacent to both $u$ and $v$ it must differ from both at all coordinates, so there are $k-2$ options for each of the coordinates where $u$ and $v$ differ and $k-1$ options for coordinates which are shared. Now, since $uv\in E(\mathbb{T}(d,k))$ if and only if $i=d$, it follows that for any $u\neq v\in V(\mathbb{T}(d,k)))$ we have   \begin{equation*}\begin{aligned} |(N_{  \mathbb{T}(d,k)}(u)\triangle N_{  \mathbb{T}(d,k)}(v))\backslash\{u,v\}\Big| & \geq 2\left((k-1)^d  - (k-2)^i(k-1)^{d-i}- \mathbf{1}(i=d)\right) \\ &\geq 2\left((k-1)^d  - (k-2)(k-1)^{d-1} \right)  \\ &=   2(k-1)^{d-1}.\end{aligned}\end{equation*} Thus by  Lemma \ref{lowerbound} we have $\tww(  \mathbb{T}(d,k))\geq 2(k-1)^{d-1} $.  
\end{proof}

Although the previous result shows that Theorem \ref{thm:tensor} is tight, one of the graphs in the tensor product has twin-width zero and so it is not clear the $\tww(G)\cdot \Delta(H)$ term in Theorem \ref{thm:tensor} is needed. However, by a very similar proof to that of Proposition \ref{prop:stronghypercubes}, one can show that for any $c,d\geq 2$ the twin-width of the tensor product of $c$ and $d$ dimensional hypercubes satisfies  \begin{equation*}2cd-4\cdot \max\{c,d\}\leq \tww(  \mathbb{H}(c,2)\times   \mathbb{H}(d,2)) \leq 2cd -2\cdot \max\{c,d\}.\end{equation*} This shows that the $\tww(G)\cdot \Delta(H)$ term in Theorem \ref{thm:tensor} is necessary.  

The following result shows that if one of the graphs in the tensor product is a star graph, where $K_{1,n}$ denotes a star on $n+1$ vertices, then we do not incur the full cost of the max degree term. 

\begin{restatable}[Tensor with a Star]{proposition}{tensorstar}\label{Prop:tensorstar}Let $G$ be any graph, then for any integer $n\geq 1$ 
	\begin{equation*}\tww(G\times K_{1,n})\leq \tww(G\times K_2) \leq \tww(G) + 2.\end{equation*} 
\end{restatable}
\begin{proof}
	Let $s$ be the centre of the star $K_{1,n}$. We do the following series of contractions to reduce $G\times K_{1,n}$ to $G\times K_2$. In each phase we choose a pair of vertices $i,j\in V(K_{1,n})\backslash \{s\}$ and then one by one we contract pairs $(u,i)$ and $(u,j)$ into a single vertex $(u,k)$ until there are no vertices $(u,i) $ or $(u,j)$ for any $u \in V(G)$ remaining. We then start the next phase, continuing like this until there are is only one vertex $i \neq s$ remaining. At this point we are left with a trigraph over $G\times K_2$. 
	We claim that during this sequence of contractions no red edges are created, the final trigraph has no red edges.
	
	To see why this holds observe that by the definition of the tensor product, for any vertex $(u,i)$ we have $N(u,i)= \{(v,s): uv \in E(G)  \}$ as $K_{1,n}$ is a star. It follows that for any pair $i,j\neq s$ of distinct leaves in $S_n$ and any $u\in V(G)$ we have $N(u,i)=N(u,j)$. Thus the pair $(u,i), (u,j)$ are twins and can be contracted without creating red edges.  Finally the inequality $ \tww(G\times K_2) \leq \tww(G) + 2$ holds by Theorem \ref{thm:tensor}.
\end{proof}

\subsection{Strong Product}

The \emph{strong product} $G \boxtimes H$ has vertex set $V(G) \times V(H)$ and $E(G \boxtimes H) = E(\cart{G}{H}) \cup E(G \times H) $, thus \begin{equation*}  \begin{aligned} & (u,v)(u',v') \in E(G \boxtimes H)\quad  \text{if and only if} \\ & \quad \qquad [u = u' \text{ and }vv' \in E(H)] \text{ or }[v=v'\text{ and }uu' \in E(G)]\text{ or }[uu' \in E(G)\text{ and }vv' \in E(H)].\end{aligned}\end{equation*} The following bound on the twin-width of a strong product is proved in~\cite{TWII}.
\begin{theorem}[{\cite[Theorem 2.7]{TWII}}]\label{thm:bonnetstrong}
	Let $G$ and $H$ be two graphs.
	Then \begin{equation*}\tww(G \boxtimes H) \leqslant \max\left\{\tww(G)(\Delta(H)+1)+\Delta(H),\; \tww(H) \right\} +\Delta(H).\end{equation*}
\end{theorem}
Unlike for tensor products this bound is not tight for strong products of complete graphs as $K_n\boxtimes K_m \cong K_{n\cdot m}$, and thus $\tww(K_n\boxtimes K_m)=0  $, for any $n,m\geq 1$. However the following result shows that Theorem \ref{thm:bonnetstrong} is tight for the strong product of two hypercubes.

\begin{proposition}[Strong Product of Hypercubes]\label{prop:stronghypercubes}For any $c,d\geq 2$ we have  \begin{equation*}\tww(  \mathbb{H}(c,2)\boxtimes   \mathbb{H}(d,2)) = 2cd +2\cdot \min\{c-d,\;d-c\} -4.\end{equation*}
\end{proposition}
\begin{proof} Denote $\mathbb{S}(c,d)= \mathbb{H}(c,2)\boxtimes   \mathbb{H}(d,2)$. We begin with the upper bound. Recall that, for any $d\geq 2$, we have $\tww(  \mathbb{H}(d,2)) = 2d-4 $ by Proposition \ref{prop:hamming} and $\Delta(  \mathbb{H}(d,2))=d$. Thus, for $c,d\geq 3$, Theorem \ref{thm:bonnetstrong} gives 
	\begin{equation*}\tww(  \mathbb{S}(c,d)) \leq  \max\big\{(2c-4)\cdot (d+1) + d, \;2d-4\big\}  +d =  (2c-4)\cdot(d+1) + 2d = 2cd +2c-2d-4, 
	\end{equation*}and as $\boxtimes $ is symmetric the bound $2cd +2d-2c-4$ also holds, as claimed. Similarly, by Theorem \ref{thm:bonnetstrong}, 
	\begin{equation*} \tww(  \mathbb{S}(2,2)) \leq  \max\big\{2, \;0\big\} + 2 =4 = 2\cdot 2\cdot 2 - 2\min\{2-2,2-2\}-4, \end{equation*}and for the final case, where w.l.o.g.\ we can assume $c=2$ and $d\geq 3$ by symmetry, we have 
	\begin{equation*} \tww(  \mathbb{S}(2,d)) \leq  \max\big\{0, \; d\big\} + d =2d = 2\cdot 2\cdot d + 2\min\{2-d,d-2\} -4,  \end{equation*} this completes the proof of the upper bound. 
	
	We now consider the lower bound which will follow by Lemma \ref{lowerbound}. Let $\overline{N}_G(v)=N_G(v)\cup \{v\}$ be the closed neighbourhood of a vertex $v\in V(G)$. It is known \cite[Exercise 4.6]{Handbook}, that for any graphs $G,H$ and vertices $u,v\in V(G)$ and $i,j\in V(H)$ we have \begin{equation*} \overline{N}_{G\boxtimes H}((u,i)) = \overline{N}_{G}(u)\times \overline{N}_{H}(i).\end{equation*} Thus, we have
	\begin{equation*}\overline{N}_{G\boxtimes H}((u,i))\cap \overline{N}_{G\boxtimes H}((v, j)) = \left(\overline{N}_{G}(u)\cap \overline{N}_{G}(v)\right)\times \left(\overline{N}_{H}(i)\cap \overline{N}_{H}(j)\right). \end{equation*}Consequently $\overline{N}_{  G\boxtimes H}((u,i))\triangle \overline{N}_{G\boxtimes H}((v,j)) $ is given by 
	\begin{equation}\label{eq:nbhdiff} \left( \left(\overline{N}_{  G}(u)\times \overline{N}_{  H}(i)\right) \cup \left(\overline{N}_{  G}(v)\times  \overline{N}_{  H}(j)\right)\right)\backslash \left( \left(\overline{N}_{G}(u)\cap \overline{N}_{G}(v)\right)\times \left(\overline{N}_{H}(i)\cap \overline{N}_{H}(j)\right)\right) .
	\end{equation}Observe also that for any $(u,i),(v,j) \in V(G\boxtimes H)$ we have \begin{equation}\label{eq:closedopen}\left(N_{G\boxtimes H}((u,i))\triangle N_{G\boxtimes H}((v,j))\right)\backslash\{(u,i),(v,j)\} = \left(\overline{N}_{G\boxtimes H}((u,i))\triangle \overline{N}_{G\boxtimes H}((v,j))\right)\backslash\{(u,i),(v,j)\}. \end{equation} 
	
	\newcommand{\cNu}{\overline{N}_{c}(u)}
	\newcommand{\cNv}{\overline{N}_{c}(v)}
	\newcommand{\dNi}{\overline{N}_{d}(i)}
	\newcommand{\dNj}{\overline{N}_{d}(j)}
	Let $B((u,i), (v,j)) = (N_{  \mathbb{S}(c,d)}((u,i))\triangle N_{  \mathbb{S}(c,d)}((v,j)))\backslash\{(u,i),(v,j)\}$ and for any integer $k\geq 1$ and denote $\overline{N}_k(u)=\overline{N}_{\mathbb{H}(k,2)}(u) $ for ease of reading. Then, by \eqref{eq:nbhdiff} and \eqref{eq:closedopen}, for any $(u,i),(v,j)$ we have
	\begin{align}\label{eq:bbdd}|B((u,i), (v,j))|&\geq \left| \cNu \times \dNi \right|  + \left| \cNv \times \dNj \right|  -  2 \left| \left( \cNu \cap \cNv \right) \times \left( \dNi \cap \dNj \right)\right|\notag \\&\qquad \qquad \qquad -|\{(u,i),(v,j)\}|\notag \\
		&= 2(c+1)(d+1)  -  2 \left| \left( \cNu \cap \cNv \right) \times \left( \dNi \cap \dNj \right)\right|-2 \notag\\
		&= 2cd +2c+2d  -  2 \left| \left( \cNu \cap \cNv \right) \times \left( \dNi \cap \dNj \right)\right| 
		. \end{align}  
	By Lemma~\ref{lowerbound}, the twin-width of $\mathbb{S}(c,d)$ is lower-bounded by the smallest value of $|B((u,i),(v,j))|$ over all these cases.  We now establish a very simple claim.  
	\begin{clm}For any $d\geq 2$ and $u,v \in V(\mathbb{H}(d,2))$ where $u\neq v$, we have $\left|\overline{N}_{d}(u) \cap \overline{N}_{d}(v)\right|  \leq 2 $ . 
	\end{clm}
	\begin{poc}
		It will help to consider $  \mathbb{H}(d,2)$ as follows: we identify each vertex as a string in $\{0,1\}^d$ where two vertices are connected if and only if their strings differ in one exactly coordinate. We have three cases: 
		
		\noindent\textbf{Case 1 [$uv\in E(  \mathbb{H}(d,2))$]:} Recall that the hypercube $\mathbb{H}(d,2)$ is bipartite, thus $u$ and $v$ cannot have any common neighbours.   
		Hence, $\overline{N}_d(u)\cap \overline{N}_d(v)= \{u,v\}$ and so $|\overline{N}_d(u)\cap \overline{N}_d(v)|=2$.  
		
		\noindent\textbf{Case 2 [$d_{  \mathbb{H}(d,2)}(u, v)=2$]:} It follows that $u$ and $v$ differ at exactly two coordinates $i$ and $j$ where $i\neq j$. Thus if $w\in V(  \mathbb{H}(d,k))$ is adjacent to both $u$ and $v$ it must be equal to $u$ at all but one coordinate and equal to $v$ at all but one coordinate, so $|\overline{N}_d(u)\cap \overline{N}_d(v)|=2$. 
		
		\noindent\textbf{Case 3 [$d_{  \mathbb{H}(d,2)}(u, v)\geq 3$]:} In this case $\overline{N}_d(u)\cap \overline{N}_d(v)=\emptyset $  thus $|\overline{N}_d(u)\cap \overline{N}_d(v)|=0\leq 2$. \end{poc}

	Having established the claim we can now check, using \eqref{eq:bbdd}, that \[|B((u,i),(v,j))|\geq 2cd +2\cdot \min\{c-d,\;d-c\} -4\] for the following three cases, which cover all possible pairs of vertices $(u,i)$ and $(v,j)$.
	
	\medskip

	\noindent\textbf{Case 1 [$u\neq v, i\neq j$]:}  Observe that in this case, by the Claim above, we have $|\cNu \cap \cNv|\leq 2$ and $|\dNi \cap \dNj | \leq 2$. Consequently, $\left| \left( \cNu \cap \cNv \right) \times \left( \dNi \cap \dNj \right)\right|\leq 4$. Thus, by \eqref{eq:bbdd} and $c,d\geq 2$,
	\[|B((u,i), (v,j))|\geq 2cd +2c+2d -  2 \cdot 4   \geq 2cd +2\cdot \min\{c-d,\;d-c\} -4. \]
	
	\noindent\textbf{Case 2 [$u=v, i \neq j$]:} In this case $\cNu =\cNv$, thus $|\cNu \cap \cNv| = c+1$. Since $i\neq j$, the Claim above gives  $|\dNi \cap \dNj | \leq 2$. Thus $\left| \left( \cNu \cap \cNv \right) \times \left( \dNi \cap \dNj \right)\right|\leq 2(c+1)$ and so by \eqref{eq:bbdd}
	\[|B((u,i), (v,j))|\geq 2cd +2c+2d -  2 \cdot 2(c+1)   \geq 2cd +2\cdot \min\{c-d,\;d-c\} -4. \]  
	\noindent\textbf{Case 3 [$u\neq v, i = j$]:} Analogously to Case 2,  $|\dNi \cap \dNj| = d+1$ and  $|\cNu \cap \cNv | \leq 2$. Thus  
	\[|B((u,i), (v,j))|\geq 2cd +2c+2d -  2 \cdot 2(d+1)   \geq 2cd +2\cdot \min\{c-d,\;d-c\} -4. \]  This establishes the lower bound by Lemma \ref{lowerbound}.  \end{proof}

\subsection{Lexicographic Product}

The lexicographic product of $G$ and $H$, written $\lex{G}{H}$, has vertex set $V(G) \times V(H)$ and for $u,v \in V(G)$ and $i,j \in V(H)$,
\begin{equation*}
	(u,i)(v,j) \in E(\lex{G}{H}) \quad  \text{if and only if}\quad [u = v \text{ and } (i,j) \in E(H)] \text{ or } [(u,v) \in E(G)]. 
\end{equation*}
We note that the following theorem is a corollary of a more general result \cite[Lemma 9]{Bonnet0RTW21} which determines the twin-width of a modular partition. However, we give an alternative proof of the result here for completeness. 
\begin{theorem}[{\cite[Lemma 9]{Bonnet0RTW21}}]\label{thm:lex} For any graphs $G,H$ we have 
	\begin{equation*}\tww\left(\lex{G}{H}\right) = \max\{\tww(G), \tww(H)\}.\end{equation*} 
\end{theorem}

\begin{proof}
	First, note that as there are induced subgraphs of $\lex{G}{H}$ that are isomorphic to both $G$ and $H$, we get that $\tww\left(\lex{G}{H}\right)\geq \max\{\tww(G), \tww(H)\}$ from Proposition~\ref{prop:induced}. It remains to prove that $\max\{ \tww(G), \tww(H)\}$ is an upper bound, which we show by giving a contraction sequence for $\lex{G}{H}$.
	
	Let $H_m, \ldots, H_1$ be a $\tww(H)$-sequence for $H$, and label the vertices of $G$ with $[n]$ in an arbitrary order.
	For a fixed $u\in [n]$, we call the $u$-th copy of $H$ the subgraph of $\lex{G}{H}$ induced by the vertex set $\{(u, i) : i\in V(H)\}$.
	We contract $\lex{G}{H}$ to $\lex{G}{K_1}$ by contracting each copy of $H$ in turn by following the sequence $H_m, \ldots, H_1$.
	That is, we first contract the first copy of $H$ to $K_1$ by following $H_m, \ldots, H_1$ and then start with the second copy of $H$.
	
	For any pair $(u, i)$ and $(v, k)$ with $u\neq v$, recall that  $(u,i)(v,k)\in E(\lex{G}{H})$ if and only if $(u,v)\in E(G)$.
	Thus, for any two vertices of the form $(u,i)$ and $(u,j)$ (i.e.,~two vertices in the same copy of $H$), if there is some $(v,k)$ such that $(u,i)(v,k)\in E(\lex{G}{H})$ but $(u,j)(v,k)\not\in E(\lex{G}{H})$, it must hold that $u=v$ (i.e., $(v,k)$ is in the same copy of $H$ as $(u,i)$ and $(u,j)$).
	This means that as we contract each copy of $H$ to $K_1$, any red edges must be in the copy of $H$, and so no vertex ever has a red degree greater than $\tww(H)$.
	Lastly, we are left with $\lex{G}{K_1}$, which is isomorphic to $G$, with no red edges, and so we can contract $G$ according to a $\tww(G)$-sequence, and it follows that the maximum red degree of the whole sequence, and an upper bound on $\tww(\lex{G}{H})$, is given by $\max\{\tww(G), \tww(H)\}$.
\end{proof}
\begin{remark}
	Since cographs are precisely the graphs of twin-width zero, this result gives another proof of the fact that the class of cographs is closed under lexicographic products.
\end{remark}

\subsection{Modular Product}
The modular product of $G$ and $H$, written $\gmod{G}{H}$, has vertex set $V(G) \times V(H)$ and for $u,v \in V(G)$, where $u\neq v$ and $i,j \in V(H)$ where $i\neq j$, we have
\begin{equation*}
	(u,i)(v,j) \in E(\gmod{G}{H}) \iff [(u,v) \in E(G) \text{ and } (i,j) \in E(H)] \text{ or } [(u,v) \not\in E(G) \text{ and } (i,j) \not\in E(H)],  
\end{equation*}furthermore $(u,i)(v,j) \notin E(\gmod{G}{H})$ whenever $u=v$ or $i=j$. Recall that $P_n$ denotes a path on $n$ vertices. 
\begin{restatable}{theorem}{mod}\label{thm:mod-lb}
	For $n\geq 5$, 
	\begin{equation*}\tww\left(\gmod{P_n}{P_n}\right)\geq n+1.\end{equation*} 
\end{restatable}
\begin{proof}
	We prove this by showing that for any pair of distinct vertices in $\gmod{P_n}{P_n}$, contracting the pair gives a trigraph with red degree at least $n+1$.
	
	We will work with $\gmod{P_n}{P'_n}$, where $P_n$ and $P'_n$ are both paths on $n$ vertices. The extra prime marker serves to distinguish the two graphs and makes the proof easier to follow.
	Label the vertices of each of $P_n$ and $P'_n$ with the elements of $[n]$ such that, for $i\in[n-1]$, vertices $i$ and $i+1$ are adjacent.
	
	Let $(u,i)$ and $(v,j)$ be the first pair of vertices contracted by any contraction sequence.
	Without loss of generality, we can say that $u\leq \frac{n}{2}$ and $u<v$.
	We note that the case $u=v$ is, by symmetry of $P_n$ and $P'_n$, absorbed by Case 1 below.
	
	\textbf{Case 1 [$i=j$]:} As $(u,i)$ and $(v,j)$ must be distinct to be contracted, it must hold that $u\neq v$, and so there must exist some vertex, say $w$, in $P_n$ such that $u\neq w$, $uw\not\in E(P_n)$ but $vw\in E(P_n)$.
	For each $k\in V(P'_n)\setminus \{i\}$, $(u,i)$ is adjacent to $(w, k)$ in $\gmod{P_n}{P'_n}$ if and only if $(v,j)$ is not adjacent to $(w,k)$, adding $n-1$ red edges incident with the new contracted vertex.
	
	Additionally, let $\ell$ be a neighbour of $i$ in $P'_n$.
	If $uv\in E(P_n)$, then $(u,i)(v,\ell)$ is in $E(\gmod{P_n}{P'_n})$ but $(v,i)(v,\ell)$ is not in $E(\gmod{P_n}{P'_n})$, and $(v,i)(u,\ell)$ is in $E(\gmod{P_n}{P'_n})$ but $(u,i)(u,\ell)$ is not in $E(\gmod{P_n}{P'_n})$, leading to two more red edges incident with the new contracted vertex.
	Alternatively, if $uv\not\in E(P_n)$, then for each $k \in V(P'_n)$ that is not a neighbour of $i$, and satisfying $k\neq i$, the edge $(u,i)(v,k)$ is in $E(\gmod{P_n}{P'_n})$ but the edge $(v,i)(v,k)$ is not in $E(\gmod{P_n}{P'_n})$. For $n\geq 5$, and for any choice of $i$, there are at least two choices for $k$ leading to at least two more red edges incident with the new contracted vertex.
	In either case, we have at least $n+1$ red edges adjacent to the new vertex.
	
	\textbf{Case 2 [$i\neq j$]:} Again there must exist a vertex, say $w$, in $P_n$ such that $u\neq w$, $uw\not\in E(P_n)$ but $vw\in E(P_n)$.
	For each $k\in V(P'_n)\setminus \{i,j\}$, $(u,i)$ is adjacent to $(w,k)$ in $\gmod{P_n}{P'_n}$ if and only if $(v,j)$ is not adjacent to $(w,k)$, adding $n-2$ red edges incident with the new contracted vertex.
	
	Similarly, there must exist a vertex, say $k$, in $P'_n$ such that $i\neq k$, $ik\not\in E(P'_n)$ but $jk\in E(P'_n)$.
	For each $x\in V(P_n)\setminus \{u,v\}$, $(u,i)$ is adjacent to $(x,k)$ in $\gmod{P_n}{P'_n}$ if and only if $(v,j)$ is not adjacent to $(x,k)$, adding a further $n-2$ red edges incident with the new contracted vertex.
	
	For $n\geq 5$, such a contraction must result in at least $2n - 4 \geq n + 1$ red edges incident with the new vertex, concluding the result.
\end{proof}

A computational search using a SAT encoding for twin-width~\cite{Schidler2022} has shown that $\tww(\gmod{P_6}{P_6}) = 9$, proving that this bound is not tight. However, since for any $n\geq 1$ the path $P_n$ has twin-width at most one and degree at most two, we obtain the following corollary to Theorem \ref{thm:mod-lb}. 

\begin{corollary}There does not exist a bounded function $f$ such that for all graphs $G,H$ we have \begin{equation*}\tww(\gmod{G}{H})\leq f(\tww(G),\tww(H),\Delta(G),\Delta(H)).\end{equation*}
\end{corollary}
\subsection{Corona Product}

The \emph{corona product} of $G$ and $H$, written $\cor{G}{H}$, has vertex set $V(G) \times (\{\infty\}\cup V(H))$ and for $u,v \in V(G)$ and $i,j \in V(H)\cup \{\infty\}$,
$(u,i)(v,j) \in E(\cor{G}{H})$ if and only if one of the following holds:
\begin{itemize}
	\item $(u,v)\in E(G)$ and $i = j=\infty$, or
	\item $(i,j)\in E(H)$ and $u = v$, or
	\item $u = v$ and $i = \infty$ and $j \neq \infty$. 
\end{itemize}

\begin{restatable}{theorem}{corprod}\label{thm:corona} For any graphs $G,H$ we have 
	\begin{equation*}\tww\left(\cor{G}{H}\right) \leq \max\{\tww(G)+1, \tww(H), 2\}.\end{equation*} 
\end{restatable}

\begin{proof}
	Let $H_m, \ldots, H_1$ be a $\tww(H)$-sequence for $H$, and label the vertices of $G$ with $[n]$ in an arbitrary order.
	For a fixed $u\in [n]$, we call the $u$-th copy of $H$ the subgraph of $\cor{G}{H}$ induced by the vertex set $\{(u, i) : i\in V(H)\}$.
	We contract $\cor{G}{H}$ by contracting each copy of $H$ in turn by following the sequence $H_m, \ldots, H_1$.
	That is, we first contract the first copy of $H$ to $K_1$ by following $H_m, \ldots, H_1$ and then start with the second copy of $H$.
	
	For any pair $(u, i)$ and $(v, k)$, where $(u,i)$ is in the $u$-th copy of $H$ and $(v, k)$ is not in the $u$-th copy of $H$ (i.e., either $k=\infty$ or $u\neq v$), recall that  $(u,i)(v,k)\in E(\cor{G}{H})$ if and only if $(v,k) = (u,\infty)$.
	But $(u,\infty)$ is adjacent to every vertex in the $u$-th copy of $H$.
	This means that as we contract each copy of $H$ to $K_1$, any red edges must be in the copy of $H$, and so no vertex ever has a red degree greater than $\tww(H)$.
	
	After this process, we are left with a copy of $G$ on vertices of the form $(u, \infty)$ with a pendant vertex connected with a black edge to each vertex.
	For each $u$, let $(u, p_u)$ be the pendant vertex adjacent to $(u, \infty)$.
	Let $G_n,\ldots,G_1$ be a $\tww(G)$-sequence for $G$, and let $\{v_{n-1},u_{n-1}\},\ldots,\{v_1,u_1\}$ be a sequence of pairs of vertices in $V(G)$ such that for $k\in [n-1]$, $G_k$ is obtained from $G_{k+1}$ by contracting vertices $v_k$ and $u_k$.
	For $k\in [n-1]$ decreasing, we contract vertices as follows.
	
	Contract vertices $(u,p_u)$ and $(v,p_v)$ to vertex $(w,p_w)$, and then contract vertices $(u,\infty)$ and $(v,\infty)$ to the vertex $(w,\infty)$.
	Contracting to $(w,p_w)$ creates a vertex of degree two with two red edges, but both $(u,\infty)$ and $(v,\infty)$ will have red degree of at most $\tww(G)+1$.
	Then when $(u,\infty)$ and $(v,\infty)$ are contracted, they will have a red edge to $(w,p_w)$ plus any red edges resulting from the $\tww(G)$-sequence for $G$, for maximum of $\max\{2, \tww(G)+1\}$. 
	
	Once these contractions are completed for all $k\in [n-1]$ we are left with a trigraph over $P_2$ which can be contracted without introducing any edges.
\end{proof}

Note that as there are induced subgraphs of $\cor{G}{H}$ that are isomorphic to both $G$ and $H$, we trivially get that $\tww\left(\cor{G}{H}\right)\geq \max\{\tww(G), \tww(H)\}$ and so the $\tww(H)$ is a necessary part of the bound.

Using a computational search~\cite{Schidler2022} to compute twin-widths, we give the following examples that demonstrate that none of the elements maximised in our bound can be reduced by a constant or removed entirely.
If $G=K_3$ and $H = K_1$ then $\tww(G) = 0$, $\tww(H) = 0$, and $\tww(\cor{G}{H}) = 2$, so the $2$ is necessary in the maximum. 
If $G$ is the Paley graph $P(9)$ and $H=K_1$ then $\tww(G) = 4$, $\tww(H) = 0$, and $\tww(\cor{G}{H}) = 5$, so $\tww(G)+1$ is necessary in the maximum.

The $l$-corona product, a repeated product defined as $G \bigcirc^1 H = G \bigcirc H$ and for an integer $\ell\geq 2$, $G \bigcirc^\ell H = (G \bigcirc^{\ell-1} H) \bigcirc H$ was recently introduced \cite{Furmanczyk-lcorona}.
For this repeated product, we derive the following from repeated application of Theorem~\ref{thm:corona}.
\begin{corollary}
	For any graphs $G$ and $H$ and integer $\ell \geq 1$ we have
	\begin{equation*}\tww(G \bigcirc^\ell H) \leq \max \{\tww(G)+\ell,\; \tww(H) + \ell - 1,\; \ell+1\}.\end{equation*}
\end{corollary}

\subsection{Rooted Product}
The \emph{rooted product} of a graph $G$ and a rooted graph $H$ with root $r$, written $\rooted{G}{H}$, has vertex set $V(G) \times V(H)$ and for $u,v \in V(G)$ and $i,j \in V(H)$,
\begin{equation*}
	(u,i)(v,j) \in E(\rooted{G}{H}) \quad  \text{if and only if}\quad[(u,v) \in E(G) \text{ and } i = j = r] \text{ or } [ u = v \text{ and } (i,j)\in E(H)]. 
\end{equation*}

\begin{restatable}{theorem}{rootedprod}\label{thm:rooted}
	For a graph $G$ and a rooted graph $H$ with root vertex $r$, 
	\begin{equation*}\tww\left(\rooted{G}{H}\right)\leq \max \{ \tww(H') + 1, d_H(r), \tww(G) + 1, 2\}\end{equation*} 
	where $H'$ is the induced subgraph of $H$ on vertex set $V(H)\setminus \{r\}$ and $d_H(r)$ is the degree of vertex $r$ in $H$.
\end{restatable}
\begin{proof}
	We show this by giving a contraction sequence that obtains this bound.
	For each $v\in V(G)$, let the $v$'th copy of $H$ be the subgraph of $\rooted{G}{H}$ induced by vertices of the form $\{ (v,i) : i\in V(H)\}$, and let the $v$'th copy of $H'$ be the induced subgraph of $\rooted{G}{H}$ on vertices of the form $\{(v,i) \mid i\in V(H) \text{ and } i\neq r\}$.
	
	We first contract each copy of $H'$ to $K_1$, then follow a $\tww(G)$-sequence of $G$ interspersed with contractions of the leaf vertices resulting from contractions of the copies of $H'$.
	
	Note that any edge in $\rooted{G}{H}$ that leaves the $v$'th copy of $H'$ must be incident with $(v, r)$, so as we contract each copy of $H'$ the maximum red degree of any vertex within the copy of $H'$ must be at most $\tww(H')+1$.
	At the same time, the maximum red degree of $(v,r)$ is at most $d_H(r)$, as no vertex of the form $(u, r)$, for any $u\in V(G)$, is involved in a contraction.
	
	After these contractions, we are left with $\rooted{G}{P_2}$ where each edge of the form $(v,i)(v,j)$ is potentially red.
	Let $G_n, \ldots, G_1$ be a $\tww(G)$-sequence for $G$, and
	let $\{v_{n-1}, u_{n-1}\}, \ldots, \{v_1, u_1\}$ be a sequence of pairs of vertices such that for $k\in [n-1]$, $G_k$ is obtained from $G_{k+1}$ by contracting vertices $v_k$ and $u_k$.
	For $k\in [n-1]$, let $(v_k,i)$ (respectively $(u_k,i)$) be the one vertex adjacent to $(v_k,r)$ (respectively $(u_k,r)$) in the same copy of $H$.
	For $k\in [n-1]$ decreasing, first contract the vertices $(u_k,i)$ and $(v_k,i)$, and then contract
	the vertices $(v_k, r)$ and $(u_k, r)$.
	Contracting $(u_k,i)$ and $(v_k, i)$ will create a vertex of degree at most 2 (and hence red degree at most 2), and the subsequent contraction of $(u_k,r)$ and $(v_k, r)$ will have red degree at most $\tww(G)+1$.
	
	The max red degree of this sequence is thus at most $\max \{\tww(H') + 1, d_  \mathbb{H}(r), \tww(G)+1, 2\}$.
\end{proof}

By taking $G=C_3$ and $H=P_2$, we can calculate $\tww(\rooted{G}{H}) = 2$ while $\tww(G)+1 = 1$ and $\tww(H') + 1 = 1$, demonstrating that the 2 is required.

\section{Bounding the Twin-width of Replacement and Zig-zag Products}\label{sec:nonstandrd} We now define the replacement and zig-zag products which have a slightly different form to the products we have seen so far, see \cite{Alon,Reingold} for more details. Both products require $G$ to be a $d$-regular graph, $H$ to be a $\delta$-regular graph on $d$ vertices, and both graphs to be supplied with a one-to-one labelling of each of the edges originating from each vertex with the labels $[d]$ and $[\delta]$ respectively. Note that such a labelling always exists for any regular graph as a single edge can have two different labels from each endpoint (i.e. this need not be a coloring). The structure of the resulting products will depend on this labelling as they will be described using the following rotation maps, which in turn depend on the labelling. 

For a $d$-regular undirected graph $G$, where $V(G)=[n]$, with a one-to-one edge labeling by $[d]$ from each vertex, the \emph{rotation map}
$\operatorname{Rot}_G : [n] \times  [d] \rightarrow [n] \times [d]$ is defined as follows: $\operatorname{Rot}_G(v, i) = (w, j)$ if the $i$-th  edge incident to $v$ leads to $w$, and this edge is the $j$-th  edge incident to $w$. Observe that $\operatorname{Rot}_G(\operatorname{Rot}_G(v, i))=(v, i) $ for any $(v, i)$, and thus $\operatorname{Rot}_G$ defines the adjacency relation in a simple undirected graph $G$.

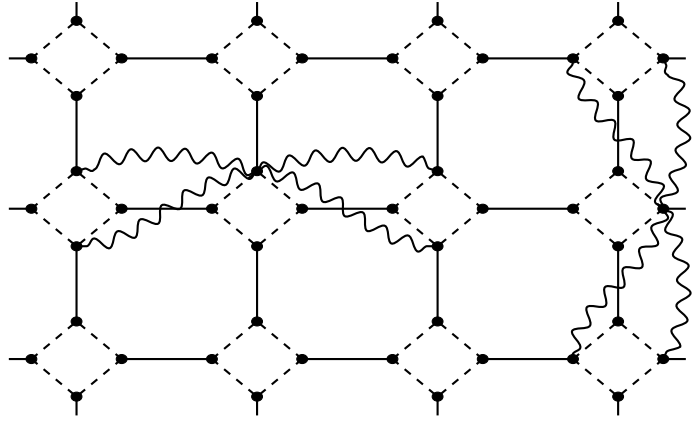
\begin{figure}[!htb]
	\center\begin{tikzpicture}[xscale=0.6,yscale=0.5]	\usetikzlibrary{arrows.meta}
		\usetikzlibrary{decorations.markings}
		\usetikzlibrary{decorations.pathreplacing}
		\foreach \x in {0,2,4,6,8,10,12,14}
		\foreach \y in {-4,0,4}{
			\draw[fill] (\x,\y) circle (.12);}
		\foreach \x in {0,4,8,12}
		\foreach \y in {-4,0,4}{
			\draw[dashed,thick] (\x,\y) --  (\x+1,\y+1);
			\draw[dashed,thick] (\x,\y) --  (\x+1,\y-1);
			\draw[dashed,thick] (\x+2,\y) --  (\x+1,\y+1);
			\draw[dashed,thick] (\x+2,\y) --  (\x+1,\y-1);}
		\foreach \x in {1,5,9,13}
		\foreach \y in {-4,0}{
			\draw[thick] (\x,\y+1) --  (\x,\y+3);}
		\foreach \x in {2,6,10}
		\foreach \y in {-4,0,4}{
			\draw[thick] (\x,\y) --  (\x+2,\y);}
		\foreach \x in {1,5,9,13}
		\foreach \y in {-5,-3,-1,1,3,5}{\draw[fill] (\x,\y) circle (.12);}
		\foreach \x in {1,5,9,13}{
			\draw[thick] (\x,5) --  (\x,5.5);
			\draw[thick] (\x,-5) --  (\x,-5.5);}
		\foreach \y in {-4,0,4}{
			\draw[thick] (-.5,\y) --  (0,\y);
			\draw[thick] (14,\y) --  (14.5,\y);}
		
		\draw[decorate, decoration = {snake},thick] (5,1) to[out=20,in=160 ] (9,1);
		\draw[decorate, decoration = {snake},thick] (5,1) to[out=-10,in=180 ] (9,-1);
		\draw[decorate, decoration = {snake},thick] (5,1) to[out=160,in=20 ] (1,1);
		\draw[decorate, decoration = {snake},thick] (5,1) to[out=190,in=0 ] (1,-1);
		
		\draw[decorate, decoration = {snake},thick] (14,0) to[out=70,in=-70 ] (14,4);
		\draw[decorate, decoration = {snake},thick] (14,0) to[out=-70,in=70 ] (14,-4);
		\draw[decorate, decoration = {snake},thick] (14,0) to[out=100,in=-90 ] (12,4);
		\draw[decorate, decoration = {snake},thick] (14,0) to[out=-100,in=90 ] (12,-4);
		
	\end{tikzpicture}
	\caption{This figure depicts the rooted and zig-zag products where $G$ is a grid, and $H$ is a $4$-cycle. The edges of $G\replace H$ are straight-solid and dotted, where the straight-solid edges correspond to edges of $G$. The edges of $G\zigzag H$ are wiggly, though most are left out to avoid over-complicating the figure.}\label{fig:repzigzag}
\end{figure}
\subsection{Replacement Product}

If $G$ is a $d$-regular graph on $[n]$ with rotation map $\operatorname{Rot}_G $ and $H$ is a $\delta$-regular graph on $[d]$ with rotation map $\operatorname{Rot}_H $ their \emph{replacement product} $G \replace H$ is the $(\delta + 1)$-regular graph on $[n] \times [d]$ whose rotation map $\operatorname{Rot}_{G \replace H}$ is given by 
\begin{equation*}\operatorname{Rot}_{G \replace H}((v,k),i) = \begin{cases}
		((v,m),j), \text{ where }(m,j)= \operatorname{Rot}_H(k,i) & \text{ if }i \leq   \delta\\
		\left(\operatorname{Rot}_{G}(v,k),i \right) &   \text{ if }i= \delta+1.
\end{cases}\end{equation*}

It will help to think of each vertex $v$ of $G$ being replaced by a \emph{cloud} $C(v)=\{(v,i) : i\in [d]\}$ of $d$ vertices. For each $u\in V(G)$ the $d$ vertices of its cloud are a copy of $H$ and for each $uv\in E(G)$ there is exactly one edge of the form $(u,i)(v,j) \in E(G \replace H)$, where the $i$ and $j$ depend on the rotation maps $\operatorname{Rot}_{G}$ and $\operatorname{Rot}_{H}$. We now prove an elementary result which will be useful when reducing the clouds.

\begin{restatable}{lemma}{spider}\label{lem:spider}Let $G$ be any graph and let $S\subseteq V(G)$ be such that $|N(s)\setminus S|\geq 1$ for all $s\in S$. Then there is a $\sum_{s\in S} |N(s)\setminus S|$-sequence reducing $S$ to a single vertex in $G$. \end{restatable}
\begin{proof}Fix any sequence of contractions for $H=G[S]$, the subgraph of $G$ induced by the set $S$. For any $v\in S$ and $t\geq 0$ let  $S_t$, and $N_t(v)$, be the remaining vertices in $S$, and the neighbourhood of $v$, after $t\geq 0$ contractions respectively. Let $u\in S_t$ be an arbitrary vertex after any $t\geq 0$ contractions, then \begin{equation}\label{eq:nbhbdd} |N_t(u)| = \left|(S_t\cap N_t(u) )\cup ( N_t(u)\setminus S_t ) \right| \leq  \left|(S_t\cap N_t(u))\cup \bigcup _{s\in S_t\setminus N_t(u)}  N_t(s)\setminus S_t  \right|,\end{equation}where the last inequality follows since $u \in   S_t\setminus N_t(u)$. Observe that for any $t$ and $s\in S_t$ we have $|N_t(s)\setminus S_t|\geq |N(s)\setminus S|\geq 1$ as no vertices outside of $S$ are contracted. Thus it follows that\begin{equation*} | S_t\cap N_t(u)| = \sum_{s\in S_t\cap N_t(u) } 1 \leq \sum_{s\in S_t\cap N_t(u)}  |N_t(s)\setminus S_t|.\end{equation*} Thus by \eqref{eq:nbhbdd} and the triangle inequality   
	\begin{equation*} |N_t(u)|\leq | S_t\cap N_t(u)|   +  \sum_{s\in S_t\setminus N_t(u)}  |N_t(s)\setminus S_t|\leq \sum_{s\in S_t }  |N_t(s)\setminus S_t|\leq \sum_{s\in S }  |N(s)\setminus S|, \end{equation*} where the last inequality holds as no vertices outside of $S$ are contracted. Thus as the degree of a vertex in $S$ never exceeds $\sum_{s\in S }  |N(s)\setminus S|$ during the contraction sequence, this is also a bound on the red degree. 
\end{proof}

Lemma \ref{lem:spider} guarantees that simply contracting each cloud of $G\replace H$ in any order gives a $\Delta(G)$-sequence leaving a trigraph over $G$. This will now be used to prove Theorem \ref{replace}, which holds for any choice of the rotation maps used to construct $G\replace H$.

\begin{restatable}{theorem}{replaceup}\label{replace}For any regular graph $G$ and any $\Delta(G)$-vertex regular graph $H$ we have \begin{equation*}\tww(G\replace H) \leq  \tww(G)   + \Delta(G).   \end{equation*} 
\end{restatable}
\begin{proof} 	
	Let $G_n, \dots, G_1$ be a $\tww(G)$-sequence for $G$. The first stage is to contract each of the `clouds', that is the graphs $H_v$ induced by the vertex sets $C(v)= \{(v,i): i\in [d]\}$ for each $v\in V(G)$. We do these one by one in any order. Since each vertex $c\in C(v)$ satisfies $|N_{G\replace H}(c)\setminus C(v)|=1 $, and each of these neighbours is distinct by construction we have $\sum_{c\in C(v) }|N_{G\replace H}(c)\setminus C(v)| = \Delta(G)$. Thus, by Lemma \ref{lem:spider}, there exists a $ \Delta(G)$-sequence reducing $C(v)$ to a single vertex (we call $v$) in $G\replace H$. Since the sets $\{C(v)\}_{v\in V(G)}$ are disjoint and there is at most one edge linking each pair we have that there is a $\Delta(G)$-sequence reducing the graph $G\replace H$ to a trigraph over $G$. 
	
	The second stage is to reduce this trigraph over $G$ using a $\tww(G)$-sequence for $G$. Thus by Lemma \ref{lem:degreetww} we have $\tww(G\replace H) \leq \max\left\{ \Delta(G) , \tww(G)+ \Delta(G) \right\}=\tww(G)+ \Delta(G)  $, as claimed. 
\end{proof}

Since most bounds on $\tww(G\star H)$ in this paper, for some product $\star$, depend on both $\tww(G)$ and $\tww(H)$ it might appear odd that the bound above does not depend on $\tww(H)$. However, recall that  $|V(H)|=\Delta(G)$ and so we have $\tww(H)\leq \Delta(G)-1$. The following lower bound is almost immediate.
\begin{lemma}\label{replacelow2}For any regular graph $G$ and any $\Delta(G)$-vertex regular graph $H$ we have \begin{equation*}\tww(G\replace H) \geq \tww( H).   \end{equation*} 
\end{lemma}
\begin{proof}
	Each cloud in $G \replace H$ induces a copy of $H$, so the bound follows from Proposition \ref{prop:induced}.
\end{proof}
We now give an example which shows Theorem \ref{replace} is tight up-to a constant factor. Let $\mathbb{F}_q$ be the field with $q$ elements. Then for  $q = 1 \text{ mod } 4$,  the \emph{Paley} graph $\operatorname{P}(q)$ has vertex set $\mathbb{F}_q$ and edge set $\left\{(a,b) : a-b =c^2 \text{ where } c\in \mathbb{F}_q \right\} $. The finite field $\mathbb{F}_q $ contains $(q-1)/2$ distinct squares, thus $\operatorname{P}(q)$ is $\frac{q-1}{2}$-regular. It is known by \cite[Theorem 1.4.]{AhnBounds} that $\tww(\operatorname{P}(q))= (q-1)/2$. For the construction, take $K_{q+1} \replace P(q) $ where $q$ satisfies $q = 1 \text{ mod } 4$, $K_{q+1}$ is a clique on $q+1$ vertices, and $P(q)$ is the Paley graph on $q$ vertices. Since  $\tww(\operatorname{P}(q))= (q-1)/2$ by \cite[Theorem 1.4.]{AhnBounds}, Lemma \ref{replacelow2} shows that Theorem \ref{replace} is tight up-to a constant factor (just slightly bigger than $2$) for the graph $K_{q+1} \replace P(q) $.

We now prove a less trivial lower bound. An interesting consequence of this  bound is that the twin-width of the replacement product of any two (non-trivial) graphs is always at least two.  
\begin{restatable}{theorem}{replacelow}\label{replacelow}For any regular graph $G$ and any $\Delta(G)$-vertex regular graph $H$ we have \begin{equation*}\tww(G\replace H) \geq \left\lceil \sqrt{2\cdot \Delta(H)}\right\rceil.   \end{equation*} 
\end{restatable}

\begin{proof}
	Let $C_v=\{(v,i) : i\in [\Delta(G)]\}$ be the cloud of $\Delta(G)$ vertices at the vertex $v\in V(G)$. Observe that if the cloud is contracted to a single vertex $w$ before there are any contractions between two vertices $x,y \in N(C_v)\setminus C_v$ then $\Delta(G)\geq \Delta(H)$ red edges must be adjacent to $w$. 
	
	We now consider the first contraction in a $\tww(G \replace H) $-sequence for $G\replace H$ where two vertices from different clouds are contracted. Let this be the $\tau$-th  contraction and observe that such a time must exist as $|V(G)|>1$. Up to this point the only contractions that could have occurred are between pairs of vertices in the same cloud. It follows that if there are any contractions before time $\tau$ then the vertices remaining have a red degree which is at least the number of vertices that have been contracted into them (since each vertex in a cloud leads to a distinct vertex in another cloud). Let $x,y$ be the two vertices,  where $x\in C_u$ and $y\in C_v$ and $u\neq v$, to be contracted at time $\tau$. After this contraction all edges from each of $x$ and $y$ to vertices within their clouds will be red. Observe that if the maximum red degree of a vertex in $x's$ cloud at time $\tau$ is $i$, then $x$ is adjacent to at least $\lceil \Delta(H)/i \rceil$ vertices in its own cloud $C_u$ at time $\tau$. Similarly, if the maximum red degree in $C_v$ at time $\tau$ is $j$ then $y$ is adjacent to at least $\lceil \Delta(H)/j \rceil$ vertices in $C_v$ at time $\tau$. When we contract $x$ and $y$ to $w$ the red degree of $w$ is at least $\lceil \Delta(H)/i \rceil + \lceil \Delta(H)/j \rceil$. 
	
	It follows that, for some $i,j\leq \Delta(G)$, we have 
	\begin{equation*}\tww(G \replace H) \geq \max \left\{\left\lceil \Delta(H)/i \right\rceil + \left\lceil \Delta(H)/j \right\rceil,\; i , \; j \right\}.  \end{equation*} We thus obtain the bound $\tww(G \replace H) \geq \left\lceil \sqrt{2\cdot \Delta(H)}\right\rceil$ by minimising over all $i$ and $j$. \end{proof}

Using our bounds for replacement products, and adapting constructions from \cite{Alon,Reingold}, we were able to construct a sequence of bounded degree expanders with poly-log twin-width\footnote{For the full construction, see Section 4.3 in the first version of this paper \href{https://arxiv.org/pdf/2202.11556v1.pdf}{ArXiv:2202.11556v1}.}. However, we do not include this result as a construction of expanders with bounded twin-width is given by \cite{TWII}.  
\subsection{Zig-Zag Product}

If $G$ is a $d$-regular graph on $[n]$ with rotation map $\operatorname{Rot}_G $ and $H$ is a $\delta$-regular graph on $[d]$ with rotation map $\operatorname{Rot}_H $ their \emph{zig-zag product} $G \zigzag H$ is the $\delta^2$-regular graph on $[n] \times [d]$ whose rotation map $\operatorname{Rot}_{G \zigzag H}$ is given by 
\begin{equation*}\operatorname{Rot}_{G \zigzag H}\left((v,k),(i,j)\right) =  
	\left((w,\ell),(j',i')\right),\end{equation*}where we find $ w,\ell,j'$ and $i'$ by the using the rotation functions for $G$ and $H$ as follows\begin{equation*} (k',i')= \operatorname{Rot}_H(k,i), \quad (w,\ell')= \operatorname{Rot}_G(v,k'), \quad \text{ and }\quad (\ell,j')= \operatorname{Rot}_H(\ell',j) . \end{equation*}
Note that in the definition of the rotation function $\operatorname{Rot}_{G \zigzag H}\left((v,k),(i,j)\right)$ we identify each $(i,j)$, where $i,j\in [d]$, with some $k\in [d^2]$.    

Similarly to the replacement product, it will help to think of each vertex of $G$ being replaced by a `cloud' of $d$ vertices, however there are a lot of differences between the edge sets of the replacement and zig-zag product graphs. Recall that in $G \replace H $ each cloud $C(v)$ induces a copy of the graph $H$, in contrast each cloud $C(v)$ in $G$ induces an independent set. One way to see the edge set of the zig-zag product is that $(u,i)$ and $(v,j)$ are connected by an edge in $G\zigzag H$ if there is a walk of length $3$ from $(u,i)$ to $(v,j)$ in $G\replace H$ where the first and third edges crossed are in clouds (correspond to edges in $H$) and the second edge corresponds to an edge in $G$, see Figure \ref{fig:repzigzag} for an example.

\begin{restatable}{theorem}{zigzagprod}\label{zigzag}For any regular graph $G$ and any $\Delta(G)$-vertex regular graph $H$ we have \begin{equation*}\tww(G\zigzag H) \leq \max\left\{ \Delta(H)^2(\Delta(G)-\Delta(H) +1) ,\; \tww(G)+ \Delta(G) \right\} .   \end{equation*} 
\end{restatable}
\begin{proof} 	
	Let $G_n, \dots, G_1$ be a $\tww(G)$-sequence for $G$ and for convenience denote $\delta=\Delta(H)$ and $d=\Delta(G)$. The first stage is to contract each of the `clouds', that is the vertex sets $C(v)=\{(v,i): i\in [d]\}$ for each $v\in V(G)$. We could just contract each of the clouds one by one, as we do in Theorem \ref{replace}, which gives a bound of $\delta^2\cdot d$ on the maximum red degree created in the first stage. However we will be a little more careful to obtain a better bound of $\delta^2\cdot(d-\delta+1)$ for the first stage.

	For the first stage order the vertices $\{v_1,\dots, v_n\}$ of $G$ arbitrarily. The clouds $\{C(v_j)\}_{j\in [n]}$ then inherit the same ordering. We will iterate through this ordering $d-1$ times, contracting a pair of vertices in the corresponding cloud at each step; more formally, at each step $j\in [n]$ of iteration $i\in [d-1] $ we contract a pair of vertices in the cloud $C(v_j)$. Thus at any iteration $i\in [d-1]$ there are at most $d-i+1$ and at least $d-i$ vertices left in each cloud. 
	
	Observe that in $G \zigzag H$ for each vertex $u\in C(v)$, where $v\in V(G)$, there are $\delta$ other vertices $v'\in V(G)$ where $vv'\in E(G)$ such that $u$ is connected to $\delta$ different vertices in each of the sets $C(v')$. Thus, for any $u\in C(v)$ and $v\in V(G)$, if there have been $i$ contractions in $C(v)$ then $u$ can be formed from merging at most $i$ vertices and so $u$ has degree at most $i\cdot \delta^2$. On the other hand as each other cloud $C(x_j)$, for $j\in [d]$ where $vx_j \in E(G)$, contains at most $d-i+1$ vertices after there have been $i$ contractions in $C(v)$ the number of edges from $u$ to each $C(x_j)$ is at most $d-i+1$. Thus the degree of $u$ is at most $i\delta \cdot (d-i+1)$. Thus, since both bounds hold the degree of $u$ is at most $\max_i\left\{ i\delta^2, i\delta (d-i+1) \right\} $, which is maximised when $i=d-\delta +1$ and thus there is a $\delta^2(d-\delta +1)$-sequence  transforming $G\zigzag H$ into a trigraph over $G$. 
	
	The second stage is to reduce this trigraph over $G$ using a $\tww(G)$-sequence for $G$. Thus by Lemma \ref{lem:degreetww} we have $\tww(G\zigzag H) \leq \max\left\{ \Delta(H)^2(\Delta(G)-\Delta(H) +1) , \tww(G)+ \Delta(G) \right\}  $, as claimed. 
\end{proof}

\section{Conclusion} \label{sec:conclude} 
In this paper we proved bounds on the twin-width of many of the most commonly known graph products. We applied these bounds to determine the twin-width of some common product graphs, in turn showing most of the bounds are tight. It is natural to ask whether some of our bounds can be improved, even for (interesting) special cases. In particular, our examples in Section \ref{sec:tensor} show that both the $\tww(G)\cdot \Delta(H)$ term and  $\Delta(H)$ terms in the bound on $\tww(G\times H)$ given by Theorem \ref{thm:tensor} are necessary. It would be nice to find an example of two graphs $G,H$ with non-zero twin-width for which Theorem \ref{thm:tensor} is tight. 

We also showed that the modular product produces an $n$ vertex graph with twin-width $\sqrt{n}$ from two paths of equal length (each of twin-width $1$). It may be interesting to see if there is any restriction one can place on graphs $G$ and $H$ so that their modular product $\gmod{G}{H}$ has bounded twin-width. 
\section*{Acknowledgements}
The authors thank Kitty Meeks and Jessica Enright for insightful discussions.
We also give our thanks to anonymous referees who reviewed an earlier version of this paper and whose feedback has been incorporated.


\begin{thebibliography}{53}
	\providecommand{\natexlab}[1]{#1}
	\providecommand{\url}[1]{\texttt{#1}}
	\expandafter\ifx\csname urlstyle\endcsname\relax
	\providecommand{\doi}[1]{doi: #1}\else
	\providecommand{\doi}{doi: \begingroup \urlstyle{rm}\Url}\fi
	
	\bibitem[Ahn et~al.(2022)Ahn, Hendrey, Kim, and Oum]{AhnBounds}
	J.~Ahn, K.~Hendrey, D.~Kim, and S.-i. Oum.
	\newblock Bounds for the twin-width of graphs.
	\newblock \emph{SIAM Journal on Discrete Mathematics}, 36\penalty0
	(3):\penalty0 2352--2366, 2022.
	\newblock \doi{10.1137/21M1452834}.
	
	\bibitem[Alon et~al.(2004)Alon, Dinur, Friedgut, and Sudakov]{MR2105948}
	N.~Alon, I.~Dinur, E.~Friedgut, and B.~Sudakov.
	\newblock Graph products, {F}ourier analysis and spectral techniques.
	\newblock \emph{Geom. Funct. Anal.}, 14\penalty0 (5):\penalty0 913--940, 2004.
	\newblock ISSN 1016-443X.
	\newblock \doi{10.1007/s00039-004-0478-3}.
	
	\bibitem[Alon et~al.(2008)Alon, Schwartz, and Shapira]{Alon}
	N.~Alon, O.~Schwartz, and A.~Shapira.
	\newblock An elementary construction of constant-degree expanders.
	\newblock \emph{Combin. Probab. Comput.}, 17\penalty0 (3):\penalty0 319--327,
	2008.
	\newblock ISSN 0963-5483.
	\newblock \doi{10.1017/S0963548307008851}.
	
	\bibitem[Balab{\'{a}}n and Hlinen{\'{y}}(2021)]{BalabanH21}
	J.~Balab{\'{a}}n and P.~Hlinen{\'{y}}.
	\newblock Twin-width is linear in the poset width.
	\newblock In P.~A. Golovach and M.~Zehavi, editors, \emph{16th International
		Symposium on Parameterized and Exact Computation, {IPEC} 2021}, volume 214 of
	\emph{LIPIcs}, pages 6:1--6:13. Schloss Dagstuhl - Leibniz-Zentrum f{\"{u}}r
	Informatik, 2021.
	\newblock \doi{10.4230/LIPIcs.IPEC.2021.6}.
	
	\bibitem[Barik et~al.(2007)Barik, Pati, and Sarma]{BarikPS07}
	S.~Barik, S.~Pati, and B.~K. Sarma.
	\newblock The spectrum of the corona of two graphs.
	\newblock \emph{{SIAM} J. Discret. Math.}, 21\penalty0 (1):\penalty0 47--56,
	2007.
	\newblock \doi{10.1137/050624029}.
	
	\bibitem[Barrow and Burstall(1976)]{BarrowMod}
	H.~G. Barrow and R.~M. Burstall.
	\newblock Subgraph isomorphism, matching relational structures and maximal
	cliques.
	\newblock \emph{Inf. Process. Lett.}, 4\penalty0 (4):\penalty0 83--84, 1976.
	\newblock \doi{10.1016/0020-0190(76)90049-1}.
	
	\bibitem[Berg{\'{e}} et~al.(2022)Berg{\'{e}}, Bonnet, and
	D{\'{e}}pr{\'{e}}s]{Hardness}
	P.~Berg{\'{e}}, {\'{E}}.~Bonnet, and H.~D{\'{e}}pr{\'{e}}s.
	\newblock Deciding twin-width at most 4 is np-complete.
	\newblock In M.~Bojanczyk, E.~Merelli, and D.~P. Woodruff, editors, \emph{49th
		International Colloquium on Automata, Languages, and Programming, {ICALP}
		2022}, volume 229 of \emph{LIPIcs}, pages 18:1--18:20, 2022.
	\newblock \doi{10.4230/LIPIcs.ICALP.2022.18}.
	
	\bibitem[Bonnet et~al.(2021{\natexlab{a}})Bonnet, Geniet, Kim, Thomass{\'{e}},
	and Watrigant]{TWII}
	{\'{E}}.~Bonnet, C.~Geniet, E.~J. Kim, S.~Thomass{\'{e}}, and R.~Watrigant.
	\newblock Twin-width {II:} small classes.
	\newblock In D.~Marx, editor, \emph{Proceedings of the 2021 {ACM-SIAM}
		Symposium on Discrete Algorithms, {SODA} 2021}, pages 1977--1996. {SIAM},
	2021{\natexlab{a}}.
	\newblock \doi{10.1137/1.9781611976465.118}.
	
	\bibitem[Bonnet et~al.(2021{\natexlab{b}})Bonnet, Geniet, Kim, Thomass{\'{e}},
	and Watrigant]{TWIII}
	{\'{E}}.~Bonnet, C.~Geniet, E.~J. Kim, S.~Thomass{\'{e}}, and R.~Watrigant.
	\newblock Twin-width {III:} max independent set, min dominating set, and
	coloring.
	\newblock In N.~Bansal, E.~Merelli, and J.~Worrell, editors, \emph{48th
		International Colloquium on Automata, Languages, and Programming, {ICALP}
		2021}, volume 198 of \emph{LIPIcs}, pages 35:1--35:20. Schloss Dagstuhl -
	Leibniz-Zentrum f{\"{u}}r Informatik, 2021{\natexlab{b}}.
	\newblock \doi{10.4230/LIPIcs.ICALP.2021.35}.
	
	\bibitem[Bonnet et~al.(2021{\natexlab{c}})Bonnet, Nesetril, de~Mendez,
	Siebertz, and Thomass{\'{e}}]{perms}
	{\'{E}}.~Bonnet, J.~Nesetril, P.~O. de~Mendez, S.~Siebertz, and
	S.~Thomass{\'{e}}.
	\newblock Twin-width and permutations.
	\newblock \emph{arXiv}, 2102.06880, 2021{\natexlab{c}}.
	
	\bibitem[Bonnet et~al.(2022{\natexlab{a}})Bonnet, Geniet, Tessera, and
	Thomass{\'{e}}]{Groups}
	{\'{E}}.~Bonnet, C.~Geniet, R.~Tessera, and S.~Thomass{\'{e}}.
	\newblock Twin-width {VII:} groups.
	\newblock \emph{arXiv}, 2204.12330, 2022{\natexlab{a}}.
	
	\bibitem[Bonnet et~al.(2022{\natexlab{b}})Bonnet, Giocanti, de~Mendez, Simon,
	Thomass{\'{e}}, and Torunczyk]{TW:IV}
	{\'{E}}.~Bonnet, U.~Giocanti, P.~O. de~Mendez, P.~Simon, S.~Thomass{\'{e}}, and
	S.~Torunczyk.
	\newblock Twin-width {IV:} ordered graphs and matrices.
	\newblock In S.~Leonardi and A.~Gupta, editors, \emph{{STOC} '22: 54th Annual
		{ACM} {SIGACT} Symposium on Theory of Computing}, pages 924--937. {ACM},
	2022{\natexlab{b}}.
	\newblock \doi{10.1145/3519935.3520037}.
	
	\bibitem[Bonnet et~al.(2022{\natexlab{c}})Bonnet, Kim, Reinald, and
	Thomass{\'{e}}]{TWVI}
	{\'{E}}.~Bonnet, E.~J. Kim, A.~Reinald, and S.~Thomass{\'{e}}.
	\newblock Twin-width {VI:} the lens of contraction sequences.
	\newblock In J.~S. Naor and N.~Buchbinder, editors, \emph{Proceedings of the
		2022 {ACM-SIAM} Symposium on Discrete Algorithms, {SODA} 2022, Virtual
		Conference / Alexandria, VA, USA, January 9 - 12, 2022}, pages 1036--1056.
	{SIAM}, 2022{\natexlab{c}}.
	\newblock \doi{10.1137/1.9781611977073.45}.
	
	\bibitem[Bonnet et~al.(2022{\natexlab{d}})Bonnet, Kim, Reinald, Thomass{\'{e}},
	and Watrigant]{Bonnet0RTW21}
	{\'{E}}.~Bonnet, E.~J. Kim, A.~Reinald, S.~Thomass{\'{e}}, and R.~Watrigant.
	\newblock Twin-width and polynomial kernels.
	\newblock \emph{Algorithmica}, 84\penalty0 (11):\penalty0 3300--3337,
	2022{\natexlab{d}}.
	\newblock \doi{10.1007/s00453-022-00965-5}.
	
	\bibitem[Bonnet et~al.(2022{\natexlab{e}})Bonnet, Kim, Thomass{\'{e}}, and
	Watrigant]{TWI}
	{\'{E}}.~Bonnet, E.~J. Kim, S.~Thomass{\'{e}}, and R.~Watrigant.
	\newblock Twin-width {I:} tractable {FO} model checking.
	\newblock \emph{J. {ACM}}, 69\penalty0 (1):\penalty0 3:1--3:46,
	2022{\natexlab{e}}.
	\newblock \doi{10.1145/3486655}.
	
	\bibitem[Brakensiek(2017)]{Brakensiek17}
	J.~Brakensiek.
	\newblock Vertex isoperimetry and independent set stability for tensor powers
	of cliques.
	\newblock In K.~Jansen, J.~D.~P. Rolim, D.~Williamson, and S.~S. Vempala,
	editors, \emph{Approximation, Randomization, and Combinatorial Optimization.
		Algorithms and Techniques, {APPROX/RANDOM} 2017}, volume~81 of \emph{LIPIcs},
	pages 33:1--33:15. Schloss Dagstuhl - Leibniz-Zentrum f{\"{u}}r Informatik,
	2017.
	\newblock \doi{10.4230/LIPIcs.APPROX-RANDOM.2017.33}.
	
	\bibitem[Chellali et~al.(2012)Chellali, Favaron, Hansberg, and
	Volkmann]{ChellaliFHV12}
	M.~Chellali, O.~Favaron, A.~Hansberg, and L.~Volkmann.
	\newblock k-domination and k-independence in graphs: {A} survey.
	\newblock \emph{Graphs Comb.}, 28\penalty0 (1):\penalty0 1--55, 2012.
	\newblock \doi{10.1007/s00373-011-1040-3}.
	
	\bibitem[Clark and Suen(2000)]{ClarkS00}
	W.~E. Clark and S.~Suen.
	\newblock Inequality related to vizing's conjecture.
	\newblock \emph{Electron. J. Comb.}, 7, 2000.
	
	\bibitem[Dreier et~al.(2022)Dreier, Gajarsk{\'{y}}, Jiang, de~Mendez, and
	Raymond]{DreierGJMR22}
	J.~Dreier, J.~Gajarsk{\'{y}}, Y.~Jiang, P.~O. de~Mendez, and J.~Raymond.
	\newblock Twin-width and generalized coloring numbers.
	\newblock \emph{Discret. Math.}, 345\penalty0 (3):\penalty0 112746, 2022.
	\newblock \doi{10.1016/j.disc.2021.112746}.
	
	\bibitem[Frucht and Harary(1970)]{Frucht1970-tk}
	R.~Frucht and F.~Harary.
	\newblock On the corona of two graphs.
	\newblock \emph{Aequationes Math.}, 4\penalty0 (3):\penalty0 322--325, Oct.
	1970.
	
	\bibitem[Furmańczyk and Zuazua(2022)]{Furmanczyk-lcorona}
	H.~Furmańczyk and R.~Zuazua.
	\newblock Adjacent vertex distinguishing total coloring of corona product of
	graphs.
	\newblock \emph{arXiv}, 2208.10884, 2022.
	
	\bibitem[Gajarsk{\'{y}} et~al.(2022)Gajarsk{\'{y}}, Pilipczuk, and
	Torunczyk]{g2021stable}
	J.~Gajarsk{\'{y}}, M.~Pilipczuk, and S.~Torunczyk.
	\newblock Stable graphs of bounded twin-width.
	\newblock In C.~Baier and D.~Fisman, editors, \emph{{LICS} '22: 37th Annual
		{ACM/IEEE} Symposium on Logic in Computer Science}, pages 39:1--39:12. {ACM},
	2022.
	\newblock \doi{10.1145/3531130.3533356}.
	
	\bibitem[Geller and Stahl(1975)]{GELLER197587}
	D.~Geller and S.~Stahl.
	\newblock The chromatic number and other functions of the lexicographic
	product.
	\newblock \emph{Journal of Combinatorial Theory, Series B}, 19\penalty0
	(1):\penalty0 87--95, 1975.
	\newblock ISSN 0095-8956.
	\newblock \doi{https://doi.org/10.1016/0095-8956(75)90076-3}.
	
	\bibitem[Ghandehari and Hatami(2008)]{MR2368031}
	M.~Ghandehari and H.~Hatami.
	\newblock Fourier analysis and large independent sets in powers of complete
	graphs.
	\newblock \emph{J. Combin. Theory Ser. B}, 98\penalty0 (1):\penalty0 164--172,
	2008.
	\newblock ISSN 0095-8956.
	\newblock \doi{10.1016/j.jctb.2007.06.003}.
	
	\bibitem[Godsil and McKay(1978)]{GodsilMcKay}
	C.~D. Godsil and B.~D. McKay.
	\newblock A new graph product and its spectrum.
	\newblock \emph{Bull. Austral. Math. Soc.}, 18\penalty0 (1):\penalty0 21--28,
	1978.
	\newblock ISSN 0004-9727.
	\newblock \doi{10.1017/S0004972700007760}.
	
	\bibitem[Greenwell and Lov\'{a}sz(1974)]{MR357199}
	D.~Greenwell and L.~Lov\'{a}sz.
	\newblock Applications of product colouring.
	\newblock \emph{Acta Math. Acad. Sci. Hungar.}, 25:\penalty0 335--340, 1974.
	\newblock ISSN 0001-5954.
	\newblock \doi{10.1007/BF01886093}.
	
	\bibitem[Gurski(2017)]{Gurski17}
	F.~Gurski.
	\newblock The behavior of clique-width under graph operations and graph
	transformations.
	\newblock \emph{Theory Comput. Syst.}, 60\penalty0 (2):\penalty0 346--376,
	2017.
	\newblock \doi{10.1007/s00224-016-9685-1}.
	
	\bibitem[Hahn and Tardif(1997)]{HahnTardif}
	G.~Hahn and C.~Tardif.
	\newblock Graph homomorphisms: structure and symmetry.
	\newblock In \emph{Graph symmetry ({M}ontreal, {PQ}, 1996)}, volume 497 of
	\emph{NATO Adv. Sci. Inst. Ser. C: Math. Phys. Sci.}, pages 107--166. Kluwer
	Acad. Publ., Dordrecht, 1997.
	\newblock \doi{10.1007/978-94-015-8937-6\_4}.
	
	\bibitem[Hammack et~al.(2011)Hammack, Imrich, and Klav\v{z}ar]{Handbook}
	R.~Hammack, W.~Imrich, and S.~Klav\v{z}ar.
	\newblock \emph{Handbook of product graphs}.
	\newblock Discrete Mathematics and its Applications (Boca Raton). CRC Press,
	Boca Raton, FL, second edition, 2011.
	\newblock ISBN 978-1-4398-1304-1.
	
	\bibitem[Hedetniemi(1966)]{Hedetniemi}
	S.~T. Hedetniemi.
	\newblock Homomorphisms of graphs and automata.
	\newblock \emph{Technical Report 03105-44-T, University of Michigan}, 1966.
	
	\bibitem[Imrich(1972)]{ImrichAssociative}
	W.~Imrich.
	\newblock Assoziative {P}rodukte von {G}raphen.
	\newblock \emph{\"{O}sterreich. Akad. Wiss. Math.-Natur. Kl. S.-B. II},
	180:\penalty0 203--239, 1972.
	\newblock ISSN 0029-8816.
	
	\bibitem[Jacob and Pilipczuk(2022)]{jacob2022bounding}
	H.~Jacob and M.~Pilipczuk.
	\newblock Bounding twin-width for bounded-treewidth graphs, planar graphs, and
	bipartite graphs.
	\newblock In M.~A. Bekos and M.~Kaufmann, editors, \emph{Graph-Theoretic
		Concepts in Computer Science - 48th International Workshop, {WG} 2022},
	volume 13453 of \emph{Lecture Notes in Computer Science}, pages 287--299.
	Springer, 2022.
	\newblock \doi{10.1007/978-3-031-15914-5\_21}.
	
	\bibitem[Jakovac(2015)]{JakovacRooted}
	M.~Jakovac.
	\newblock The k-path vertex cover of rooted product graphs.
	\newblock \emph{Discret. Appl. Math.}, 187:\penalty0 111--119, 2015.
	\newblock \doi{10.1016/j.dam.2015.02.018}.
	
	\bibitem[Klavzar and Zmazek(1996)]{KlavzarVizConj}
	S.~Klavzar and B.~Zmazek.
	\newblock On a vizing-like conjecture for direct product graphs.
	\newblock \emph{Discret. Math.}, 156\penalty0 (1-3):\penalty0 243--246, 1996.
	\newblock \doi{10.1016/0012-365X(96)00032-5}.
	
	\bibitem[Koh et~al.(1980)Koh, Rogers, and Tan]{KohRooted}
	K.~M. Koh, D.~G. Rogers, and T.~Tan.
	\newblock Products of graceful trees.
	\newblock \emph{Discret. Math.}, 31\penalty0 (3):\penalty0 279--292, 1980.
	\newblock \doi{10.1016/0012-365X(80)90139-9}.
	
	\bibitem[Kratsch et~al.(2022)Kratsch, Nelles, and Simon]{kratsch2022triangle}
	S.~Kratsch, F.~Nelles, and A.~Simon.
	\newblock On triangle counting parameterized by twin-width.
	\newblock \emph{arXiv}, 2202.06708, 2022.
	
	\bibitem[Levi(1972)]{LeviMod}
	G.~Levi.
	\newblock A note on the derivation of maximal common subgraphs of two directed
	or undirected graphs.
	\newblock \emph{Calcolo}, 9:\penalty0 341--352 (1973), 1972.
	\newblock ISSN 0008-0624.
	\newblock \doi{10.1007/BF02575586}.
	
	\bibitem[Mohan et~al.(2017)Mohan, Geetha, and Somasundaram]{MohanGS17}
	S.~Mohan, J.~Geetha, and K.~Somasundaram.
	\newblock Total coloring of corona product of two graphs.
	\newblock \emph{Australas. {J} Comb.}, 68:\penalty0 15--22, 2017.
	
	\bibitem[Pike and Sanaei(2012)]{PikeMod}
	D.~A. Pike and A.~Sanaei.
	\newblock The modular product and existential closure.
	\newblock \emph{Australas. {J} Comb.}, 52:\penalty0 173--184, 2012.
	
	\bibitem[Pilipczuk and Sokolowski(2023)]{pilipczuk2022graphs}
	M.~Pilipczuk and M.~Sokolowski.
	\newblock Graphs of bounded twin-width are quasi-polynomially
	\emph{{\(\chi\)}}-bounded.
	\newblock \emph{J. Comb. Theory, Ser. {B}}, 161:\penalty0 382--406, 2023.
	\newblock \doi{10.1016/j.jctb.2023.02.006}.
	
	\bibitem[Pilipczuk et~al.(2022)Pilipczuk, Soko{\l}owski, and
	Zych-Pawlewicz]{pilipczuk2021compact}
	M.~Pilipczuk, M.~Soko{\l}owski, and A.~Zych-Pawlewicz.
	\newblock {Compact Representation for Matrices of Bounded Twin-Width}.
	\newblock In P.~Berenbrink and B.~Monmege, editors, \emph{39th International
		Symposium on Theoretical Aspects of Computer Science (STACS 2022)}, volume
	219 of \emph{Leibniz International Proceedings in Informatics (LIPIcs)},
	pages 52:1--52:14, 2022.
	\newblock \doi{10.4230/LIPIcs.STACS.2022.52}.
	
	\bibitem[Przybyszewski(2022)]{przybyszewski}
	W.~Przybyszewski.
	\newblock Distal combinatorial tools for graphs of bounded twin-width. 	\newblock \emph{arXiv}, 2202.04006, 2022.
	
	\bibitem[Raymond et~al.(2002)Raymond, Gardiner, and Willett]{RaymondMod}
	J.~W. Raymond, E.~J. Gardiner, and P.~Willett.
	\newblock {RASCAL:} calculation of graph similarity using maximum common edge
	subgraphs.
	\newblock \emph{Comput. J.}, 45\penalty0 (6):\penalty0 631--644, 2002.
	\newblock \doi{10.1093/comjnl/45.6.631}.
	
	\bibitem[Reingold et~al.(2002)Reingold, Vadhan, and Wigderson]{Reingold}
	O.~Reingold, S.~Vadhan, and A.~Wigderson.
	\newblock Entropy waves, the zig-zag graph product, and new constant-degree
	expanders.
	\newblock \emph{Ann. of Math. (2)}, 155\penalty0 (1):\penalty0 157--187, 2002.
	\newblock ISSN 0003-486X.
	\newblock \doi{10.2307/3062153}.
	
	\bibitem[Rosenfeld(2010)]{RosenfeldRooted}
	V.~R. Rosenfeld.
	\newblock The independence polynomial of rooted products of graphs.
	\newblock \emph{Discret. Appl. Math.}, 158\penalty0 (5):\penalty0 551--558,
	2010.
	\newblock \doi{10.1016/j.dam.2009.10.009}.
	
	\bibitem[Sabidussi(1957)]{Sabidussi}
	G.~Sabidussi.
	\newblock Graphs with given group and given graph-theoretical properties.
	\newblock \emph{Canadian J. Math.}, 9:\penalty0 515--525, 1957.
	\newblock ISSN 0008-414X.
	\newblock \doi{10.4153/CJM-1957-060-7}.
	
	\bibitem[Schidler and Szeider(2022)]{Schidler2022}
	A.~Schidler and S.~Szeider.
	\newblock A {SAT} approach to twin-width.
	\newblock In \emph{2022 Proceedings of the Symposium on Algorithm Engineering
		and Experiments ({ALENEX})}, pages 67--77. Society for Industrial and Applied
	Mathematics, Jan. 2022.
	\newblock \doi{10.1137/1.9781611977042.6}.
	
	\bibitem[Shao and Solis{-}Oba(2013)]{ShaoMod}
	Z.~Shao and R.~Solis{-}Oba.
	\newblock L(2, 1)-labelings on the modular product of two graphs.
	\newblock \emph{Theor. Comput. Sci.}, 487:\penalty0 74--81, 2013.
	\newblock \doi{10.1016/j.tcs.2013.02.002}.
	
	\bibitem[Shitov(2019)]{Shitov}
	Y.~Shitov.
	\newblock Counterexamples to {H}edetniemi's conjecture.
	\newblock \emph{Ann. of Math. (2)}, 190\penalty0 (2):\penalty0 663--667, 2019.
	\newblock ISSN 0003-486X.
	\newblock \doi{10.4007/annals.2019.190.2.6}.
	
	\bibitem[Simon and Toru\'{n}czyk(2021)]{simon2021ordered}
	P.~Simon and S.~Toru\'{n}czyk.
	\newblock Ordered graphs of bounded twin-width.
	\newblock \emph{arXiv}, 2102.06881, 2021.
	
	\bibitem[Suen and Tarr(2012)]{SuenT12}
	S.~Suen and J.~Tarr.
	\newblock An improved inequality related to {V}izing's conjecture.
	\newblock \emph{Electron. J. Comb.}, 19\penalty0 (1):\penalty0 P8, 2012.
	
	\bibitem[Vizing(1968)]{Vizing}
	V.~G. Vizing.
	\newblock Some unsolved problems in graph theory.
	\newblock \emph{Uspehi Mat. Nauk}, 23\penalty0 (6 (144)):\penalty0 117--134,
	1968.
	\newblock ISSN 0042-1316.
	
	\bibitem[Vizing(1974)]{VizingMod}
	V.~G. Vizing.
	\newblock Reduction of the problem of isomorphism and isomorphic entrance to
	the task of finding the nondensity of a graph.
	\newblock \emph{Proc. 3rd All-Union Conf. Problems of Theoretical Cybernetics},
	page 124, 1974.
	
\end{thebibliography}
\end{document}